\documentclass[reqno,11pt]{article}
\usepackage{a4wide}
\usepackage{color,eucal,enumerate,mathrsfs}
\usepackage[normalem]{ulem}
\usepackage{mdwlist}
\usepackage{amsmath}
\usepackage{amssymb,epsfig,bbm}
\numberwithin{equation}{section}
\usepackage{hyperref}
\usepackage{hyphenat}
\def\Xint#1{\mathchoice
{\XXint\displaystyle\textstyle{#1}}%
{\XXint\textstyle\scriptstyle{#1}}%
{\XXint\scriptstyle\scriptscriptstyle{#1}}%
{\XXint\scriptscriptstyle%
\scriptscriptstyle{#1}}%
\!\int}
\def\XXint#1#2#3{{\setbox0=\hbox{$#1{#2#3}{%
\int}$ }
\vcenter{\hbox{$#2#3$ }}\kern-.6\wd0}}

\def\dashint{\Xint-}

\usepackage[latin1]{inputenc}

\newcommand{\N}{\mathbb{N}}

\newcommand{\R}{\mathbb{R}}

\newcommand{\mm}{{\mbox{\boldmath$m$}}}

\newcommand{\sfd}{{\sf d}}

\newcommand{\Kliminf}{K\kern-3pt-\kern-2pt\mathop{\rm lim\,inf}\limits}  % Kuratowski liminf di insiemi
\newcommand{\Lip}{\mathop{\rm Lip}\nolimits}          %Lipschitz constant
\renewcommand{\d}{{\mathrm d}}

\newcommand{\restr}[1]{\lower3pt\hbox{$|_{#1}$}}
\newcommand{\eps}{\varepsilon}  
\newcommand{\nchi}{{\raise.3ex\hbox{$\chi$}}}

\newcommand{\fr}{\penalty-20\null\hfill$\blacksquare$}                      %quadratino nero alla fine del remark, se non vi piace, la cosa migliore e' `svuotare' la macro, cosi' non bisogna intervenire sul testo

\renewcommand{\mm}{\mathfrak m}                                %misura di riferimento

\newcommand{\lip}{{\rm lip}}
\newcommand{\Comp}{{\rm Comp}}
\newcommand{\X}{{\rm X}}
\newcommand{\Y}{{\rm Y}}

\newenvironment{proof}{\removelastskip\par\medskip   % inizio e fine dimostrazione
\noindent{\textit{Proof.}}\rm}{\penalty-20\null\hfill$\square$\par\medbreak}

\newtheorem{theorem}{Theorem}[section]

\newtheorem{corollary}[theorem]{Corollary}
\newtheorem{lemma}[theorem]{Lemma}
\newtheorem{proposition}[theorem]{Proposition}

\newtheorem{definition}[theorem]{Definition}

\newtheorem{remark}[theorem]{Remark}

\newcommand{\beq}{\begin{equation}}
\newcommand{\eeq}{\end{equation}}

\title{Equivalence of two different notions of tangent bundle on rectifiable metric measure spaces}
\author{Nicola Gigli\thanks{SISSA, ngigli@sissa.it} \and Enrico Pasqualetto\thanks{SISSA, epasqual@sissa.it}}

\begin{document}
\maketitle
\begin{abstract} 
We prove that for a suitable class of metric measure spaces, the abstract notion of tangent module as defined by the first author can be isometrically identified with the space of $L^2$-sections of the `Gromov-Hausdorff tangent bundle'.

The class of spaces $(\X,\sfd,\mm)$ we consider are PI spaces  that for every $\eps>0$ admit a countable collection of Borel sets $(U_i)$ covering $\mm$-a.e.\ $\X$ and corresponding $(1+\eps)$-biLipschitz maps $\varphi_i:U_i\to\R^{k_i}$ such that $(\varphi_i)_*\mm\restr{U_i}\ll\mathcal L^{k_i}$. This class is known to contain ${\sf RCD}^*(K,N)$ spaces.

Part of the work we carry out is that to give a meaning to notion of $L^2$-sections of the Gromov-Hausdorff tangent bundle, in particular explaining what it means to have a measurable map assigning to $\mm$-a.e.\ $x\in \X$ an element of the pointed-Gromov-Hausdorff limit of the blow-up of $\X$ at $x$. 
\end{abstract}
\tableofcontents
\section*{Introduction}
\addcontentsline{toc}{section}{Introduction} In the context of metric geometry there is a well established notion of tangent space at a point: the  pointed-Gromov-Hausdorff limit of the blow-up of the space at the chosen point, whenever such limit exists. More recently, the first author proposed  in \cite{Gigli14} an abstract definition of tangent bundle to a generic metric measure space, such notion being based on the concepts of $L^\infty$-module and Sobolev functions. 

It is then natural to ask  whether there is any relation between these two notions and pretty easy to realize that without some regularity assumption on the space  there is no hope to find any: on one hand in general the study of Sobolev functions might lead to no information about the metric structure of the space under consideration  (this is the case, for instance, of spaces admitting no non-constant Lipschitz curves), on the other the pointed-Gromov-Hausdorff limits of the blow-ups can fail to exist at every point.

We restrict the attention to the class of \emph{strongly $\mm$-rectifiable} spaces $(\X,\sfd,\mm)$, defined as those spaces such that for every $\eps>0$ there exists a sequence of Borel sets $(U_i)$ covering $\mm$-a.e.\ $\X$ and maps $\varphi_i:U_i\to\R^{k_i}$ such that for every $i$
\[
\varphi_i\mbox{ is }(1+\varepsilon)\mbox{-biLipschitz with its image,} \qquad\text{and}\qquad(\varphi_i)_* (\mathfrak{m}\restr{U_i})\ll\mathcal{L}^{k_i}\restr{\varphi_i(U_i)}.
\]
The main result of this paper, Theorem  \ref{thm_L^2(TX)=L^2(T_GH_X)}, states that for this class of spaces the two notions of tangent spaces mentioned above are intimately connected, in the sense that the pointed-Gromov-Hausdorff limits of rescaled spaces are Euclidean spaces for $\mm$-a.e.\ $x\in\X$ and the space of measurable sections of the bundle formed by the collection of these blow-ups is in isometric bijection with the elements of the tangent module.

Looking for an analogy,  one might think at this result as a kind of Rademacher's theorem: in either case when defining a notion of differentiability/tangent space there is on one side a `concrete' and `geometric' notion obtained by `blow-ups' and on the other an `abstract' and `analytic' notion obtained by looking at  `weak' derivatives. For general functions/spaces these might be very different, but under appropriate regularity assumptions (Lipschitz/strongly $\mm$-rectifiable) they a.e.\ coincide.

\bigskip

The motivating example of strongly $\mm$-rectifiable space are ${\sf RCD}^*(K,N)$ spaces. The existence of $(1+\eps)$-biLipschitz charts was obtained by Mondino-Naber in \cite{Mondino-Naber14} and the fact that those maps send the reference measure into something absolutely continuous w.r.t.\ the Lebesgue one has been independently proved by Kell-Mondino in \cite{MK16} and by the authors in \cite{GP16-2} (in both cases relying on the recent powerful results of De Philippis-Rindler \cite{DPR}). 

\bigskip

Finally, we remark that part of our efforts here are made to give a meaning to the concept of `measurable sections of the bundle formed by the collection of blow-ups'. Let us illustrate the point with an example.

Suppose that we have a metric space $(\X,\sfd)$ such that for every $x\in\X$ the tangent space at $x$ in the sense of pointed-Gromov-Hausdorff limit is the Euclidean space of a certain fixed dimension $k$. Then obviously all such tangent spaces would be isometric and we might want to identify all of them with a given, fixed $\R^k$. Once this identifications are chosen, given $x\in\X$ and ${\sf v}\in \R^k$ we might think at ${\sf v}$ as an element of the tangent space at $x$ and thus a vector field should be thought of as a map from $\X$ to $\R^k$. However, the choice of the identifications/isometries of the abstract tangent spaces with the fixed $\R^k$ is highly arbitrary and affects the structure that one is building: this is better seen if one wonders what it is, say, a Lipschitz vector field, or a continuous, or a measurable one. In fact, in general there is no answer to such questions, in the sense that there is no canonical choice of these identifications: the problem is that, by the very definition, a pointed-Gromov-Hausdorff limit is the isometric class of a  metric space, rather than a `concrete' one.

As we shall see, the situation changes if one works on a strongly $\mm$-rectifiable metric measure space: much like in the smooth setting the charts of a manifold are used to give structure to the tangent bundle, in this case the presence of charts allows for a canonical identification of the tangent spaces while also ensuring  existence and uniqueness of a measurable structure on the resulting bundle (and in general nothing more than this, so that we still can't define continuous vector fields). The construction of such measurable bundle, which we call Gromov-Hausdorff tangent bundle and denote by $T_{\rm GH}\X$, is done in Section \ref{subsect_GH_Tangent_Bundle}, while in Section \ref{sect_Geometric_Interpretation_of_T_GH_X} we show that its fibres are the pointed-Gromov-Hausdorff limits of rescaled spaces, thus justifying the terminology.

\bigskip

\noindent{\bf Acknowledgment }

\noindent This research has been supported by the MIUR SIR-grant `Nonsmooth Differential Geometry' (RBSI147UG4).

\section{Preliminaries}
\subsection{Metric Measure Spaces}
For the purpose of this paper, a \textsl{metric measure space} is a triple $(\mbox{X},\mathsf{d},\mathfrak{m})$, where
\begin{equation}\label{f_mms}
\begin{array}{cl}
(\mbox{X},\mathsf{d})&\mbox{is a complete and separable metric space,}\\
\mathfrak{m}&\mbox{is a non-negative Radon measure on }\mbox{X}\mbox{.}
\end{array}
\end{equation}
Given two metric measure spaces $({\rm X},\mathsf{d}_{\rm X},\mathfrak{m}_{\rm X})$ and $({\rm Y},\mathsf{d}_{\rm Y},\mathfrak{m}_{\rm Y})$, we will always implicitly endow the product space ${\rm X}\times{\rm Y}$ with the product distance $\mathsf{d}_{\rm X}\times\mathsf{d}_{\rm Y}$ given by
\[
(\mathsf{d}_{\rm X}\times\mathsf{d}_{\rm Y})\big((x_1,y_1),(x_2,y_2)\big)^2:=\sfd_{\rm X}(x_1,x_2)^2+\sfd_{\rm Y}(y_1,y_2)^2
\]
and the product measure $\mathfrak{m}_{\rm X}\otimes\mathfrak{m}_{\rm Y}$.\\
A metric measure space $(\mbox{X},\mathsf{d},\mathfrak{m})$ is said to be \textsl{doubling} provided there exists $C>0$ such that
\begin{equation}\label{f_doubling}
0<\mathfrak{m}\big(B_{2r}(x)\big)\leq C\,\mathfrak{m}\big(B_r(x)\big)<+\infty
\quad\mbox{ for every }x\in\mbox{X}\mbox{ and }r>0
\end{equation}
and the least such constant $C$ is called the doubling constant of the space.

A fundamental property of doubling metric measure spaces is the Lebesgue differentiation theorem, whose proof can be found e.g. in \cite{Heinonen01}:
\begin{theorem}[Lebesgue differentiation theorem]\label{thm_leb_diff_thm}
Let $(\emph{X},\mathsf{d},\mathfrak{m})$ be a doubling metric measure space. Let $f\in L^1_{\emph{loc}}(\emph{X})$. Then
\begin{equation}\label{f_leb_diff_thm}
f(x)=\lim_{r\to 0}\;\dashint_{B_r(x)}f\,\emph{d}\mathfrak{m}\quad\mbox{ for }\mathfrak{m}\mbox{-a.e. }x\in\emph{X}\mbox{.}
\end{equation}
\end{theorem}
Given a point $x\in\mbox{X}$ and a Borel subset $E$ of $\mbox{X}$, we say that $x$ is \textsl{of density} $\lambda\in[0,1]$ for $E$ if
\begin{equation}\label{f_density_pt}
D_E(x):=\lim_{r\to 0}\frac{\mathfrak{m}\big(E\cap B_r(x)\big)}{\mathfrak{m}\big(B_r(x)\big)}=\lambda\mbox{.}
\end{equation}
By applying Theorem \ref{thm_leb_diff_thm} to the function $\chi_E$, we deduce that
\begin{equation}\label{f_a.e._pt_1-density}
D_E(x)=1\quad\mbox{ for }\mathfrak{m}\mbox{-a.e. }x\in E\mbox{.}
\end{equation}
Let us recall the following well-known result about points of density 1:
\begin{lemma}\label{lemma_pts_1-density}
Let $(\emph{X},\mathsf{d},\mathfrak{m})$ be a doubling metric measure space.
Let $E\in\mathscr{B}(\emph{X})$. Let $\bar{x}\in\emph{X}$ be a point of density $1$ for $E$.
Then
\begin{equation}\label{f_pts_1-density_1}
\forall\varepsilon>0\;\;\exists\,r>0:\;\;\forall{}x\in{}B_{r}(\bar{x})\;\;\exists\,y\in{}E:\quad\mathsf{d}(x,y)<\varepsilon\,\mathsf{d}(x,\bar{x})\mbox{.}
\end{equation}
\end{lemma}
\begin{proof}
We argue by contradiction: assume the existence of $\varepsilon>0$ and of points $\{x_r\}_{r>0}\subseteq\mbox{X}$ with $\mathsf{d}(x_r,\bar{x})<r$ for every $r>0$, such that
\begin{equation}\label{f_pts_1-density_2}
E\cap B_{\varepsilon\,\mathsf{d}(x_{r},\bar{x})}(x_{r})=\emptyset\quad\mbox{ for every }r>0\mbox{.}
\end{equation}
Let $C>0$ be the doubling constant of the space. Fix $n\in\mathbb{N}$ such that $2^n\,\varepsilon\geq 2+\varepsilon$.
Thus $B_{\varepsilon\,\mathsf{d}(x_r,\bar{x})}(x_r)\subseteq B_{(1+\varepsilon)\,\mathsf{d}(x_r,\bar{x})}(\bar{x})
\subseteq B_{(2+\varepsilon)\,\mathsf{d}(x_r,\bar{x})}(x_r)\subseteq B_{2^n\,\varepsilon\,\mathsf{d}(x_r,\bar{x})}(x_r)$
for every $r>0$, hence the doubling condition implies that
\begin{equation}\label{f_pts_1-density_3}
\mathfrak{m}\big(B_{\varepsilon\,\mathsf{d}(x_{r},\bar{x})}(x_{r})\big)
\geq\frac{\mathfrak{m}\big(B_{2^{n}\varepsilon\,\mathsf{d}(x_{r},\bar{x})}(x_{r})\big)}{C^{n}}
\geq\frac{\mathfrak{m}\big(B_{(1+\varepsilon)\,\mathsf{d}(x_{r},\bar{x})}(\bar{x})\big)}{C^{n}}
\quad\mbox{ for every }r>0\mbox{.}
\end{equation}
Therefore
\[\begin{split}
D_{E}(\bar{x})&=\lim_{r\to 0}\frac{\mathfrak{m}\big(B_{(1+\varepsilon)\,\mathsf{d}(x_r,\bar{x})}(\bar{x})\cap E\big)}
{\mathfrak{m}\big(B_{(1+\varepsilon)\,\mathsf{d}(x_r,\bar{x})}(\bar{x})\big)}\\
(\text{by }\eqref{f_pts_1-density_2})\qquad&\leq
\underset{r\to 0}{\underline{\lim}}\,\frac{\mathfrak{m}\big(B_{(1+\varepsilon)\,\mathsf{d}(x_r,\bar{x})}(\bar{x})\setminus B_{\varepsilon\,\mathsf{d}(x_r,\bar{x})}(x_r)\big)}{\mathfrak{m}\big(B_{(1+\varepsilon)
\,\mathsf{d}(x_r,\bar{x})}(\bar{x})\big)}\\
&=\underset{r\to 0}{\underline{\lim}}\,\frac{\mathfrak{m}\big(B_{(1+\varepsilon)\,\mathsf{d}(x_r,\bar{x})}(\bar{x})\big)-\mathfrak{m}\big(B_{\varepsilon\,\mathsf{d}(x_r,\bar{x})}(x_r)\big)}{\mathfrak{m}\big(B_{(1+\varepsilon)\,\mathsf{d}(x_r,\bar{x})}(\bar{x})\big)}\\
(\text{by }\eqref{f_pts_1-density_3})\qquad&\leq 1-\frac{1}{C^n}<1\mbox{,}
\end{split}\]
which contradicts our assumption $D_{E}(\bar{x})=1$. 
\end{proof}
In addition to doubling spaces, another class of metric measure spaces we shall deal with is that of PI  space, whose definition requires some preliminary terminology.

Given a complete and separable metric space $(\mbox{X},\mathsf{d})$, we say that a curve $\gamma\in C\big([0,1],\mbox{X}\big)$ is \textsl{absolutely continuous} if there exists $f\in L^1(0,1)$ such that
\begin{equation}\label{f_curve_AC}
\mathsf{d}(\gamma_t,\gamma_s)\leq\int_t^s f(r)\,\mbox{d}r\quad\mbox{ for every }t,s\in [0,1]\mbox{ with }t<s\mbox{.}
\end{equation}
We will denote by $AC\big([0,1],\mbox{X}\big)$ the set of all the absolutely continuous curves in $\mbox{X}$. Given any $\gamma\in AC\big([0,1],\mbox{X}\big)$, the limit
\begin{equation}\label{f_metric_speed}
|\dot{\gamma}_t|:=\lim_{h\to 0}\frac{\mathsf{d}(\gamma_{t+h},\gamma_t)}{|h|}
\end{equation}
exists for $\mathcal{L}^1$-a.e. $t\in [0,1]$ and defines an $L^1$ function. Such map, called \textsl{metric speed} of $\gamma$, is the minimal (in the a.e. sense) $L^1$ function which can be chosen as $f$ in the right hand side of \eqref{f_curve_AC}. For a proof of these results, we refer to Theorem 1.1.2 of \cite{AmbrosioGigliSavare08}.

Given a map $u:\,\mbox{X}\to\mathbb{R}$ and a Borel function $G:\,\mbox{X}\to [0,+\infty]$, we say that $G$ is an \textsl{upper gradient} of $u$ provided
\begin{equation}\label{f_upper_gradient}
\big|u(\gamma_1)-u(\gamma_0)\big|\leq\int_0^1 G(\gamma_t)\,|\dot{\gamma}_t|\,\mbox{d}t
\quad\mbox{ for every }\gamma\in AC\big([0,1],\mbox{X}\big)\mbox{.}
\end{equation}
We are now ready to introduce the concept of PI space:
\begin{definition}[Poincar\'{e} inequality and PI spaces]\label{def_poincare_space}
Let $(\emph{X},\mathsf{d},\mathfrak{m})$ be a metric measure space. We say that  $(\emph{X},\mathsf{d},\mathfrak{m})$ \emph{supports a weak $(1,2)$-Poincar\'{e} inequality} if there exist constants $C>0$ and $\lambda\geq 1$ such that the following condition is satisfied: for every $u:{\rm X}\to\R$ continuous and  every upper gradient $G$ of $u$, it holds that
\begin{equation}\label{f_poincare_ineq}
\dashint_{B_r(x)}{\big|u-u_{B_r(x)}\big|}\,\emph{d}\mathfrak{m}\leq
C\,r\,\left(\dashint_{B_{\lambda r}(x)}G^2\,\emph{d}\mathfrak{m}\right)^{1/2}
\end{equation}
for every $x\in\emph{X}$ and $r>0$, where $u_{B_r(x)}:=\dashint_{B_r(x)}u\,\emph{d}\mathfrak{m}$.

A space which is both doubling and supporting a weak  $(1,2)$-Poincar\'{e} inequality will be called {\rm PI space.}
\end{definition}
\subsection{Lipschitz Functions}\label{subsect_Properties_of_Lipschitz_Functions}
Let $(\mbox{X},\mathsf{d}_{\mbox{\scriptsize{X}}})$ and $(\mbox{Y},\mathsf{d}_{\mbox{\scriptsize{Y}}})$ be metric spaces. A function $f:\,\mbox{X}\to\mbox{Y}$ is said to be \textsl{Lipschitz} (or, more precisely, \textsl{$\lambda$-Lipschitz}) if there exists $\lambda\geq 0$ such that $\mathsf{d}_{\mbox{\scriptsize{Y}}}\big(f(x),f(y)\big)\leq\lambda\,\mathsf{d}_{\mbox{\scriptsize{X}}}(x,y)$ for every $x,y\in\mbox{X}$. 
The smallest $\lambda\geq 0$ such that $f$ is $\lambda$-Lipschitz is denoted by $\mbox{Lip}(f)$ and is called \textsl{Lipschitz constant} of $f$.
The family of all the Lipschitz functions from $\mbox{X}$ to $\mbox{Y}$ is denoted by $\mbox{LIP}(\mbox{X},\mbox{Y})$.
For the sake of brevity, we shall write $\mbox{LIP}(\mbox{X})$ instead of $\mbox{LIP}(\mbox{X},\mathbb{R})$.
We say that a function $f:\,\mbox{X}\to\mbox{Y}$ is \textsl{$\lambda$-biLipschitz} if it is invertible and $f$, $f^{-1}$ are $\lambda$-Lipschitz.
\begin{definition}[Local Lipschitz constant]\label{def_loc_lip_const}
Let $(\emph{X},\mathsf{d})$ be a metric space. Let $f\in\emph{LIP}(\emph{X})$. Then the \emph{local Lipschitz constant} of $f$ is the function $\emph{lip}(f):\,\emph{X}\to[0,+\infty)$, which is defined by $\emph{lip}(f)(x):=0$ if $x\in\emph{X}$ is an isolated point and by
\begin{equation}\label{f_loc_lip_const}
\emph{lip}(f)(x):=\underset{\substack{y\to x \vspace{.03cm} \\ y\in\emph{X}\setminus\{x\}}}{\overline{\lim}}\frac{\big|f(y)-f(x)\big|}{\mathsf{d}(y,x)}\quad\mbox{ if }x\in\emph{X}\mbox{ is an accumulation point.}
\end{equation}
\end{definition}
One can easily prove that $\mbox{lip}(f)\leq\mbox{Lip}(f)$ and that
\begin{equation}\label{f_chain_rule_loc_lip_const_1}
\lip(f\circ\varphi)\leq\Lip(\varphi)\,\lip(f)\circ\varphi\mbox{.}
\end{equation}
for any couple of metric spaces ${\rm X},{\rm Y}$ and functions  $\varphi\in{\rm LIP}({\rm X},{\rm Y})$ and $f\in{\rm LIP}({\rm Y})$. 

Given a metric space $(\mbox{X},\mathsf{d})$, a Lipschitz function $f\in\mbox{LIP}(\mbox{X})$ and a Borel set $E\in\mathscr{B}(\mbox{X})$, we have that $\mbox{lip}\big({f|}_E\big)(x)\leq\mbox{lip}(f)(x)$ is satisfied for every point $x\in\mbox{X}$, where
$\mbox{lip}\big({f|}_E\big)$ is taken in the metric space $\big(E,{\mathsf{d}|}_{E\times E}\big)$. Simple examples show that in general equality does not hold;
however, if we restrict to the case of a doubling metric measure space, then  Lemma \ref{lemma_pts_1-density} grants that equality holds at least on density points of $E$:
\begin{proposition}\label{prop_lip(f|E)=lip(f)}
Let $(\emph{X},\mathsf{d},\mathfrak{m})$ be a doubling metric measure space. Fix a Borel set $E\in\mathscr{B}(\emph{X})$ and a Lipschitz function $f\in\emph{LIP}(\emph{X})$. Then
\begin{equation}\label{f_lip(f|E)=lip(f)_1}
\emph{lip}\big({f|}_E\big)(x)=\emph{lip}(f)(x)\quad\mbox{ for }\mathfrak{m}\mbox{-a.e. }x\in E\mbox{.}
\end{equation}
\end{proposition}
\begin{proof}
It suffices to prove that $\mbox{lip}(f)(x)\leq\mbox{lip}\big({f|}_E\big)(x)$ for every point $x\in E$ of density $1$. Thus fix $x\in E$ with $D_E(x)=1$. If $x$ is an isolated point in $X$, then $\mbox{lip}(f)(x)=\mbox{lip}\big({f|}_E\big)(x)=0$. If $x$ is an accumulation point, then take a sequence $(x_n)_n\subseteq\mbox{X}\setminus\{x\}$ converging to $x$. Up to passing to a suitable subsequence, we can assume that $\overline{\lim}_n\,\big|f(x_n)-f(x)\big|/\mathsf{d}(x_n,x)$ is actually a limit.
Moreover, possibly passing to a further subsequence, Lemma \ref{lemma_pts_1-density} provides the existence of a sequence $(y_n)_n\subseteq E$ satisfying $\mathsf{d}(x_n,y_n)<\mathsf{d}(x_n,x)/n$ for every $n\geq 1$.
In particular, $\lim_n y_n=x$ and $y_n\neq x$ for every $n\geq 1$. Therefore
\[\begin{split}
\lim_{n\to\infty}\frac{\big|f\big(x_{n})-f(x)\big|}{\mathsf{d}(x_{n},x)}&
\leq\underset{n\to\infty}{\overline{\lim}}\frac{\big|f(x_n)-f(y_n)\big|}{\mathsf{d}(x_n,y_n)}\,
\frac{\mathsf{d}(x_n,y_n)}{\mathsf{d}(x_n,x)}+\underset{n\to\infty}{\overline{\lim}}\frac{\big|f(y_n)-f(x)\big|}{\mathsf{d}(y_n,x)}\,\frac{\mathsf{d}(y_n,x)}{\mathsf{d}(x_n,x)}\\
&\leq\mbox{Lip}(f)\lim_{n\to\infty}\frac{1}{n}+
\underset{n\to\infty}{\overline{\lim}}\frac{\big|f(y_n)-f(x)\big|}{\mathsf{d}(y_n,x)}
\lim_{n\to\infty}\left(1+\frac{1}{n}\right)\\
&\leq\mbox{lip}\big({f|}_E\big)(x)\mbox{.}
\end{split}\]
The arbitrariness of $(x_n)$ gives the conclusion.
\end{proof}
In what follows we shall frequently use the fact that
\begin{equation}\label{f_McShane-Whitney}
\begin{split}
&\mbox{Given a metric space }(\mbox{X},\mathsf{d})\mbox{, a subset }E\mbox{ of }\mbox{X}\mbox{ and }f\in\mbox{LIP}(E)\mbox{,}\\
&\mbox{there exists }\bar{f}\in\mbox{LIP}(\mbox{X})\mbox{ such that }{\bar{f}|}_E=f\mbox{ and }\mbox{Lip}(\bar{f})=\mbox{Lip}(f)\mbox{.}\end{split}
\end{equation}
An explicit expression - called McShane extension - for such a $\bar f$ is given by $\bar{f}(x):=\inf\big\{f(y)+\mbox{Lip}(f)\,\mathsf{d}(x,y)\,\big|\,y\in E\big\}$ for every $x\in\mbox{X}$.

Arguing componentwise, from this fact we also directly deduce that
\[
\begin{split}
&\mbox{Given a metric space }(\mbox{X},\mathsf{d})\mbox{, a subset }E\mbox{ of }\mbox{X}\mbox{ and }f\in\mbox{LIP}(E,\R^n)\mbox{,}\\
&\mbox{there exists }\bar{f}\in\mbox{LIP}(\mbox{X},\R^n)\mbox{ such that }{\bar{f}|}_E=f\mbox{ and }\mbox{Lip}(\bar{f})\leq \sqrt n\,\mbox{Lip}(f)\mbox{.}
\end{split}
\]
Let us briefly discuss the case of Lipschitz functions from $\R^k$ into itself.  Let ${\rm End}(\R^k)$ be the set of linear maps from $\R^k$ to itself, $E\subset \R^k$ be Borel and $f:E\to\R^k$ be a Lipschitz function. Find a Lipschitz extension $\tilde f$ of $f$ to the whole $\R^k$ and use Rademacher theorem to obtain that $\tilde f$ is differentiable $\mathcal L^k$-a.e.. Call $\d \tilde f(x)\in {\rm End}(\R^k)$  such differential at the point $x$, whenever it is defined. Then it is not hard to check, for instance following the same arguments used for the proof of Proposition \ref{prop_lip(f|E)=lip(f)}, that for $\mathcal L^k$-a.e.\ $x\in E$, the value of $\d\tilde f(x)$ does not depend on the chosen extension $\tilde f$, so that the formula
\[
\d f(x):=\d\tilde f(x)\qquad\mathcal L^k-a.e.\ x\in E,
\]
is well posed and defines an element of $L^\infty(E,{\rm End}(\R^k))$ satisfying
\[
\|\d f(x)\|\leq \Lip(f)\qquad\mathcal L^k-a.e.\ x\in E.
\]
\subsection{Sobolev Calculus}
The scope of this section is to recall how to build the Sobolev space $W^{1,2}(\mbox{X})$ on a metric measure space. The following definitions and results are taken from \cite{AmbrosioGigliSavare11-2} and \cite{Gigli12}.

Let $(\mbox{X},\mathsf{d},\mathfrak{m})$ be a metric measure space, which will be fixed for the whole section. For every $t\in [0,1]$, we denote by $\mbox{e}_t:\,C\big([0,1],\mbox{X}\big)\to\mbox{X}$ the \textsl{evaluation map} at time $t$, namely
\begin{equation}\label{f_evaluation_map}
\mbox{e}_t(\gamma):=\gamma_t\quad\mbox{ for every }\gamma\in C\big([0,1],\mbox{X}\big)\mbox{.}
\end{equation}
Recall that $C\big([0,1],\mbox{X}\big)$ is a metric space, with respect to the sup distance. Hence we can consider a Borel probability measure $\boldsymbol{\pi}$ on $C\big([0,1],\mbox{X}\big)$. We say that $\boldsymbol{\pi}$ is a \textsl{test plan} provided
\begin{equation}\label{f_test_plan}\begin{split}
&(\mbox{e}_t)_\sharp\boldsymbol{\pi}\leq C\,\mathfrak{m}\;\mbox{ for every }t\in[0,1]\mbox{,}\quad\mbox{ for some constant }C>0\mbox{,}\\
&\int\!\!\!\int_0^1{|\dot{\gamma}_t|}^2\,\mbox{d}t\,\mbox{d}\boldsymbol{\pi}(\gamma)<+\infty\mbox{,}\quad\mbox{ where }\int_0^1{|\dot{\gamma}_t|}^2\,\mbox{d}t:=+\infty\mbox{ if }\gamma\notin AC\big([0,1],\mbox{X}\big)\mbox{.}
\end{split}\end{equation}
In particular, any test plan must be necessarily concentrated on $AC\big([0,1],\mbox{X}\big)$.
\begin{definition}[Sobolev class]\label{def_sobolev_class}
The \emph{Sobolev class} $\emph{S}^2(\emph{X})$ (resp. $\emph{S}^2_{\emph{loc}}(\emph{X})$) is the space of all the Borel maps $f:\,\emph{X}\to\mathbb{R}$ such that there exists $G\in L^2(\mathfrak{m})$ (resp. $G\in L^2_{\emph{loc}}(\mathfrak{m})$) satisfying
\begin{equation}\label{f_sobolev_class}
\int\big|f(\gamma_1)-f(\gamma_0)\big|\,\emph{d}\boldsymbol{\pi}(\gamma)\leq\int\!\!\!\int_0^1 G(\gamma_t)\,|\dot{\gamma}_t|\,\emph{d}t\,\emph{d}\boldsymbol{\pi}(\gamma)\quad\mbox{ for every test plan }\boldsymbol{\pi}\mbox{.}
\end{equation}
\end{definition}
Here and in what follows, $L^2_{\rm loc}(\mm)$ is the space of functions which for every $x\in\X$ coincide with some function in $L^2(\mm)$ on some neighbourhood of $x$. Similarly for other spaces.

Given $f\in\mbox{S}^2(\mbox{X})$, it is possible to prove that there exists a minimal function $|\mbox{D}f|$ in the $\mathfrak{m}$-a.e. sense which can be chosen as $G$ in \eqref{f_sobolev_class}. We call $|\mbox{D}f|$ the \textsl{minimal weak upper gradient} of $f$.

The main calculus properties of minimal weak upper gradients are the following:\\
\textsc{Locality.} If $f,g\in\mbox{S}^2_{\mbox{\scriptsize{loc}}}(\mbox{X})$ and $N\in\mathscr{B}(\mathbb{R})$ satisfies $\mathcal{L}^1(N)=0$, then
\begin{equation}\label{f_sobolev_func_1}
\begin{array}{ll}
|\mbox{D}f|=0\\
|\mbox{D}f|=|\mbox{D}g|
\end{array}\quad\begin{array}{ll}
\mathfrak{m}\mbox{-a.e. in }f^{-1}(N)\mbox{,}\\
\mathfrak{m}\mbox{-a.e. in }\{f=g\}\mbox{.}
\end{array}
\end{equation}
\textsc{Lower semicontinuity.} Let $(f_n)_n\subseteq\mbox{S}^2(\mbox{X})$ satisfy $\lim_n f_n(x)=f(x)$ for $\mathfrak{m}$-a.e. $x\in\mbox{X}$, for some $f:\,\mbox{X}\to\mathbb{R}$. Assume that $|\mbox{D}f_n|\rightharpoonup G$ weakly in $L^2(\mathfrak{m})$ as $n\to\infty$, for some $G\in L^2(\mathfrak{m})$. Then $f\in\mbox{S}^2(\mbox{X})$ and
\begin{equation}\label{f_sobolev_func_2}
|\mbox{D}f|\leq G\quad\mathfrak{m}\mbox{-a.e. in }\mbox{X}\mbox{.}
\end{equation}
\textsc{Subadditivity.} If $f,g\in\mbox{S}^2_{\mbox{\scriptsize{loc}}}(\mbox{X})$ and $\alpha,\beta\in\mathbb{R}$, then $\alpha f+\beta g\in {\rm S}^2_{\mbox{\scriptsize{loc}}}(\X)$ and
\begin{equation}\label{f_sobolev_func_3}
\big|\mbox{D}(\alpha f+\beta g)\big|\leq|\alpha||\mbox{D}f|+|\beta||\mbox{D}g|\quad\mathfrak{m}\mbox{-a.e. in }\mbox{X}\mbox{.}
\end{equation}
\textsc{Leibniz rule.} If $f,g\in\mbox{S}^2_{\mbox{\scriptsize{loc}}}(\mbox{X})\cap L^\infty_{\mbox{\scriptsize{loc}}}(\mathfrak{m})$, then $fg\in \mbox{S}^2_{\mbox{\scriptsize{loc}}}(\mbox{X})\cap L^\infty_{\mbox{\scriptsize{loc}}}(\mathfrak{m})$ and 
\begin{equation}\label{f_sobolev_funct_4}
\big|\mbox{D}(fg)\big|\leq|f||\mbox{D}g|+|g||\mbox{D}f|\quad\mathfrak{m}\mbox{-a.e. in }\mbox{X}\mbox{.}
\end{equation}
\textsc{Chain rule.} Let $f\in\mbox{S}^2_{\mbox{\scriptsize{loc}}}(\mbox{X})$ and $\varphi\in\mbox{LIP}(\mathbb{R})$. Then $\varphi\circ f\in\mbox{S}^2_{\mbox{\scriptsize{loc}}}(\mbox{X})$ and
\begin{equation}\label{f_sobolev_func_5}
\big|\mbox{D}(\varphi\circ f)\big|=|\varphi'|\circ f\,|\mbox{D}f|\quad\mathfrak{m}\mbox{-a.e. in }\mbox{X}\mbox{,}
\end{equation}
where $|\varphi'|\circ f$ is arbitrarily defined at the non-differentiability points of $\varphi$.
Notice that for $f\in \mbox{LIP}(\mbox{X})$, we trivially have that \eqref{f_sobolev_class} is satisfied for $G:=\lip(f)$, so that $f\in {\rm S}^2_{\rm loc}(\X)$ and
\begin{equation}\label{f_Lip_is_Sob_1}
|\mbox{D}f|\leq\mbox{lip}(f)\quad\mathfrak{m}\mbox{-a.e. in }\mbox{X}\mbox{.}
\end{equation}
It is a remarkable result of Cheeger (\cite{Cheeger00}) that on PI spaces the equality holds:
\begin{equation}\label{f_Lip_is_Sob_2}
(\mbox{X},\mathsf{d},\mathfrak{m})\mbox{ PI space, }f\in \mbox{LIP}(\mbox{X})\quad\Longrightarrow\quad|\mbox{D}f|=\mbox{lip}(f)\;\;\mathfrak{m}\mbox{-a.e. in }\mbox{X}\mbox{.}
\end{equation}
The Sobolev space $W^{1,2}(\X)$ is defined as
\begin{equation}\label{f_sobolev_space_1}
W^{1,2}(\emph{X}):={\rm S}^2({\X})\cap L^2(\mathfrak{m})\mbox{.}
\end{equation}
It turns out that  $W^{1,2}(\mbox{X})$ is a Banach space if endowed with the norm
\begin{equation}\label{f_sobolev_space_2}
{\|f\|}_{W^{1,2}(\mbox{\scriptsize{X}})}:=\sqrt{{\|f\|}^2_{L^2(\mathfrak{m})}+{\big\||\mbox{D}f|\big\|}^2_{L^2(\mathfrak{m})}}\quad\mbox{ for every }f\in W^{1,2}(\mbox{X})\mbox{.}
\end{equation}
We conclude recalling that
\begin{equation}
\label{eq:lipdense}
(\X,\sfd,\mm)\text{ doubling}\qquad\Rightarrow\qquad W^{1,2}\cap {\rm LIP}(\X)\text{ is dense in }W^{1,2}(\X).
\end{equation}
This non-trivial result, which in fact only requires the doubling property of the distance, has been proved in \cite{ACM14}.
%but in general it is not a Hilbert space. We say that $(\mbox{X},\mathsf{d},\mathfrak{m})$ is \textsl{infinitesimally Hilbertian} provided $W^{1,2}(\mbox{X})$ is a Hilbert space.
%
%
%
%
%\begin{theorem}[Lusin density of Lipschitz maps in $W^{1,2}(\rm{X})$]\label{thm_Lusin_density_of_Lip}
%Assume that $(\emph{X},\mathsf{d},\mathfrak{m})$ is doubling and Poincar\'{e}. Let $f\in W^{1,2}(\emph{X})$ and $\varepsilon>0$. Then
%\begin{equation}\label{f_Lusin_density_of_Lip}
%\mbox{there exists }g\in\emph{LIP}(\emph{X})\mbox{ such that }\mathfrak{m}\big(\{f\neq g\}\big)<\varepsilon\mbox{.}
%\end{equation}
%\end{theorem}
%
%
%
%
\subsection{Cotangent and Tangent Modules}
Here we recall some definitions and concepts introduced by the first author in  \cite{Gigli14}, referring to \cite{Gigli14} and \cite{Gigli16} for a more detailed discussion.

Let $(\mbox{X},\mathsf{d},\mathfrak{m})$ be a metric measure space, which will be fixed throughout the whole section.
We first give the definition of $L^2(\mathfrak{m})$-normed $L^\infty(\mathfrak{m})$-module:
\begin{definition}[$L^2(\mathfrak{m})$-normed $L^\infty(\mathfrak{m})$-module]\label{def_L2_normed_module}
Let $\mathscr{M}$ be a Banach space. Then $\mathscr{M}$ is said to be an \emph{$L^2(\mathfrak{m})$-normed $L^\infty(\mathfrak{m})$-module} provided it is endowed with a bilinear map $L^\infty(\mathfrak{m})\times\mathscr{M}\ni(f,v)\mapsto f v\in\mathscr{M}$, called \emph{multiplication}, and a function $|\cdot|:\,\mathscr{M}\to L^2(\mathfrak{m})^+$, called \emph{pointwise norm}, which satisfy the following properties:
\begin{itemize}
\item[\emph{(i)}] $(fg) v=f(g v)$ for every $v\in\mathscr{M}$ and $f,g\in L^\infty(\mathfrak{m})$.
\item[\emph{(ii)}] ${\bf 1}  v=v$ for every $v\in\mathscr{M}$, where ${\bf 1}\in L^\infty(\mm)$ is the function identically 1.
\item[\emph{(iii)}] ${\big\||v|\big\|}_{L^2(\mathfrak{m})}={\|v\|}_\mathscr{M}$ for every $v\in\mathscr{M}$.
\item[\emph{(iv)}] $|f v|=|f|\,|v|\;\;\mathfrak{m}$-a.e. in $\emph{X}$, for every $v\in\mathscr{M}$ and $f\in L^\infty(\mathfrak{m})$.
\end{itemize}
\end{definition}
Given a Borel set $A\in\mathscr{B}(\mbox{X})$, we define the `restriction'  $\mathscr M\restr A$  of $\mathscr M$ to $A$  as
\begin{equation}\label{f_restr_module}
{\mathscr{M}}\restr A:=\big\{v\in\mathscr{M}\,\big|\,\nchi_{A^c}\cdot v=0\big\}\mbox{.}
\end{equation}
Notice that ${\mathscr{M}}\restr A$ inherits the structure of $L^2(\mathfrak{m})$-normed $L^\infty(\mathfrak{m})$-module.

\bigskip

Given two $L^2(\mathfrak{m})$-normed $L^\infty(\mathfrak{m})$-modules $\mathscr{M}$ and $\mathscr{N}$, we say that a map $T:\,\mathscr{M}\to\mathscr{N}$ is a \emph{module morphism} provided it is linear continuous and it satisfies
\begin{equation}\label{f_module_morphism}
T(f v)=f \,T(v)\quad\mbox{ for every }v\in\mathscr{M}\mbox{ and }f\in L^\infty(\mathfrak{m})\mbox{.}
\end{equation}\vspace{-0.8cm}
\begin{definition}[Dual module]\label{def_dual_module}
Let $\mathscr{M}$ be an $L^2(\mathfrak{m})$-normed $L^\infty(\mathfrak{m})$-module. Then we define the \emph{dual module} $\mathscr{M}^*$ of $\mathscr{M}$ as the family of all linear continuous maps $T:\,\mathscr{M}\to L^1(\mathfrak{m})$ such that $T(f v)=f\,T(v)$ holds $\mathfrak{m}$-a.e. in $\rm{X}$ for any $v\in\mathscr{M}$ and $f\in L^\infty(\mathfrak{m})$.
\end{definition}
$\mathscr{M}^*$ naturally comes with the  structure of $L^2(\mathfrak{m})$-normed $L^\infty(\mathfrak{m})$-module:
it is a Banach space with respect to the pointwise vector operations and the operator norm, while the multiplication $f T$ between $f\in L^\infty(\mathfrak{m})$ and $T\in\mathscr{M}^*$ is defined as
\begin{equation}\label{f_dual_multiplication}
(f T)(v):=f\,T(v)\;\;\mathfrak{m}\mbox{-a.e. in }{\rm X},\quad\mbox{ for every }v\in\mathscr{M}
\end{equation}
and the pointwise norm $|T|_*$ of $T\in\mathscr{M}^*$ is given by
\begin{equation}\label{f_dual_ptwse_norm}
|T|_*:=\underset{\substack{v\in\mathscr{M}\mbox{,} \\ |v|\leq 1\;\mathfrak{m}\mbox{\scriptsize{-a.e.}}}}{\mbox{ess}\;\mbox{sup}}
\big|T(v)\big|\quad\mbox{ for every }T\in\mathscr{M}^*\mbox{.}
\end{equation}
We recall the notion of local dimension:
\begin{definition}[Local dimension of normed modules]\label{def_local_dim_module}
Let $\mathscr{M}$ be an $L^2(\mathfrak{m})$-normed $L^\infty(\mathfrak{m})$-module. Let $A\in\mathscr{B}(\emph{X})$ be such that $\mathfrak{m}(A)>0$. Then:
\begin{itemize}
\item[\emph{(i)}] Finitely many elements $v_1,\ldots,v_n\in\mathscr{M}$ are said to be \emph{independent} on $A$ provided for any $f_1,\ldots,f_n\in L^\infty(\mathfrak{m})$ it holds that
\begin{equation}\label{f_independence_in_modules}
\nchi_A\sum_{i=1}^n\,f_i v_i=0\quad\Longleftrightarrow\quad f_i=0\;\;\mathfrak{m}\mbox{-a.e. in }A,\quad\mbox{ for every }i=1,\ldots,n\mbox{.}
\end{equation}
\item[\emph{(ii)}] We say that a set $S\subset \mathscr M$ \emph{generates} $\mathscr M\restr A$ provided $\mathscr M\restr A$ is the closure of the set of finite sums of objects of the form $\nchi_A fv$ for $f\in L^\infty(\mm)$ and $v\in S$.
\item[\emph{(iii)}] We say that some elements $v_1,\ldots,v_n\in\mathscr{M}$ constitute a \emph{basis} for $\mathscr{M}\restr A$ if they are independent on $A$ and generate $\mathscr M\restr A$.
\item[\emph{(iv)}] The \emph{local dimension} of $\mathscr{M}$ on $A$ is defined to be equal to $n\in\N$ if $\mathscr{M}$ admits a basis of cardinality $n$ on $A$.
\end{itemize}
\end{definition}
Observe that the notion of local dimension is well-defined, in the sense that two different bases for $\mathscr{M}$ on $A$ must necessarily have the same cardinality. 
\bigskip

By using the language of $L^2(\mathfrak{m})$-normed $L^\infty(\mathfrak{m})$-modules described so far, we can now introduce the cotangent module $L^2(T^*\rm{X})$ associated to $(\rm{X},\mathsf{d},\mathfrak{m})$. Its definition is based upon the following result, whose proof can be found in \cite{Gigli16}:
\begin{theorem}\label{thm_characterization_cotangent_module}
There exists (up to unique isomorphism) a unique couple $(\mathscr{M},\rm{d})$, where $\mathscr{M}$ is an $L^2(\mathfrak{m})$-normed $L^\infty(\mathfrak{m})$-module and $\rm{d}:\,W^{1,2}(\rm{X})\to\mathscr{M}$ is a linear map, such that
\begin{itemize}
\item[\rm(i)] $|\emph{d}f|=|\emph{D}f|$ holds $\mathfrak{m}$-a.e. in $\emph{X}$, for every $f\in W^{1,2}(\rm{X})$,
\item[\rm(ii)] $\big\{\emph{d}f\,:\,f\in W^{1,2}(\rm{X})\big\}$ generates $\mathscr{M}$ on $\rm{X}$.
\end{itemize}
Namely, if two couples $(\mathscr{M},\rm{d})$ and $(\mathscr{M}',\rm{d}')$ as above fulfill both (i) and (ii), then there exists a unique module isomorphism $\Phi:\,\mathscr{M}\to\mathscr{M}'$ such that $\Phi\circ\rm{d}=\rm{d}'$.
\end{theorem}
\begin{definition}[Cotangent module and differential]\label{def_cotangent_module} The module provided by the previous theorem is called \emph{cotangent module} and denoted by $L^2(T^*\rm{X})$; its elements are called \emph{$1$-forms} on $\rm{X}$. The map $\d$ will be called \emph{differential}.
\end{definition}
The tangent module is then introduced by duality:
\begin{definition}[Tangent module]\label{def_tangent_module} We call \emph{tangent module} the dual of $L^2(T^*\X)$ and denote it by $L^2(T\rm{X})$.  Its elements are called \emph{vector fields} on $\rm{X}$.
\end{definition}
In order to keep consistency with the smooth setting, we shall indicate by $|\cdot|_*$ and $|\cdot|$ the pointwise norms of $L^2(T^*\rm{X})$ and $L^2(T\rm{X})$, respectively, even if we defined $L^2(T\rm{X})$ as the dual of $L^2(T^*\rm{X})$ and not vice versa.
\bigskip

We conclude the section discussing the case of $\X=\R^k$. Let us denote by $L^2(\R^k,\R^k)$ the standard space of $L^2$ vector fields on $\R^k$ and by $L^2(\R^k,(\R^k)^*)$ its dual, i.e.\ the space of $L^2$ 1-forms. Notice that the dual of $L^2(\R^k,(\R^k)^*)$ is $L^2(\R^k,\R^k)$.

We know that the Sobolev space $W^{1,2}(\R^k)$ as defined here coincides with the classically defined one via distributional derivatives and that for $f\in W^{1,2}(\R^k)$ if we consider  its distributional differential, which for a moment we denote $\hat\d f$ and which naturally belongs to $L^2(\R^k,(\R^k)^*)$, we have that its norm $|\hat\d f|$ coincides with the minimal weak upper gradient $|{\rm D}f|$ (see \cite{AmbrosioGigliSavare11}). Also, it is readily verified that 1-forms of the kind $\sum_{i=1}^n\nchi_{A_i}\hat \d f_i$, for $n\in\N$, $(A_i)$ a partition of $\R^k$ and $(f_i)\subset W^{1,2}(\R^k)$, are dense in $L^2(\R^k,(\R^k)^*)$.  Thanks to Theorem \ref{thm_characterization_cotangent_module}, these facts are sufficient to conclude that the `concrete' space of $L^2$ 1-forms $L^2(\R^k,(\R^k)^*)$ and the abstract cotangent module $L^2(T^*\R^k)$ can be canonically identified by an isomorphism which sends $\hat \d f$ to $\d f$.

Once this identification is done, it  follows that also the space of $L^2$ vector fields $L^2(\R^k,\R^k)$  can be canonically identified with the tangent module $L^2(T\R^k)$. Such identification also allows us to identify,  for a given Borel set $E\subset \R^k$, the restricted  module $L^2(T\R^k)\restr E$ with the space $L^2(E,\R^k)$.
\section{Maps of Bounded Deformation}
Consider two metric measure spaces $(\mbox{X},\mathsf{d}_{\mbox{\scriptsize{X}}},\mathfrak{m}_{\mbox{\scriptsize{X}}})$ and $(\mbox{Y},\mathsf{d}_{\mbox{\scriptsize{Y}}},\mathfrak{m}_{\mbox{\scriptsize{Y}}})$.
In Section 2.4 of \cite{Gigli14}, it is described how the notions of pullback of $1$-forms and of differential can be
built for a special class of mappings between $\mbox{X}$ and $\mbox{Y}$, which are said to be of bounded deformation.

Here we generalize those constructions to mappings that are not defined on the whole $\mbox{X}$, but only on some Borel subset $E$ of $\mbox{X}$.
More precisely, we shall first define what it means for a map $\varphi:\,E\to\mbox{Y}$ to be of bounded deformation.
Then we will prove that, under some additional assumptions on the spaces $\mbox{X}$ and $\mbox{Y}$, the function $\varphi$ canonically induces a pullback operator $\varphi^*:\,L^2(T^*\mbox{Y})\to{L^2(T^*\mbox{X})|}_E$. Finally, whenever $\varphi$ is invertible with its image and its inverse is of bounded deformation, also the differential $\mbox{d}\varphi:\,{L^2(T\mbox{X})|}_E\to L^2(T\mbox{Y})$ is well-defined.
The necessity of requiring further hypotheses on $\mbox{X}$ and $\mbox{Y}$ is due to some technical complications, which will become soon clear in the proof of Proposition \ref{prop_Sobolev_chain_rule_for_Lipschitz_maps}.

We start with the following definition:
\begin{definition}[Map of bounded deformation]\label{def_map_bdd_deform}
Let $(\emph{X},\mathsf{d}_{\emph{X}},\mathfrak{m}_{\emph{X}})$, $(\emph{Y},\mathsf{d}_{\emph{Y}},\mathfrak{m}_{\emph{Y}})$ be metric measure spaces. Fix $E\in\mathscr{B}(\emph{X})$.
Then a map $\varphi:\,E\to\emph{Y}$ is said to be \emph{of bounded deformation} if it is Lipschitz and satisfies
\begin{equation}\label{f_bdd_deform}
{\varphi}_{*}(\mathfrak{m}_{\emph{X}}\restr E)\leq{}C\,\mathfrak{m}_{\emph{Y}}\quad
\mbox{ for a suitable constant }C>0\mbox{.}
\end{equation}
The least such $C$ is called \emph{compression constant} and denoted by ${\rm Comp}(\varphi)$.
\end{definition}
The first property of this class of maps is consequence
of what discussed in Section \ref{subsect_Properties_of_Lipschitz_Functions}:
\begin{proposition}\label{prop_Sobolev_chain_rule_for_Lipschitz_maps}
Let $(\emph{X},\mathsf{d}_{\emph{X}},\mathfrak{m}_{\emph{X}})$ be a doubling space and $(\emph{Y},\mathsf{d}_{\emph{Y}},\mathfrak{m}_{\emph{Y}})$ be a PI  space.
Also, let $E\in\mathscr{B}(\emph{X})$ and let  $\varphi:\,E\to\emph{Y}$ be of bounded deformation.
Then for every $f\in\emph{LIP}(\emph{Y})$ it holds that
\begin{equation}\label{f_Sobolev_chain_rule_for_Lipschitz_maps}
|\emph{D}g|\leq\emph{Lip}(\varphi)\,|\emph{D}f|\circ\varphi\;\;\;\mathfrak{m}_{\emph{X}}\mbox{-a.e. in }E\mbox{, }\quad
\mbox{ for every }g\in\emph{LIP}(\emph{X})\mbox{ with }{g|}_{E}=f\circ\varphi\mbox{.}
\end{equation}
\end{proposition}
\begin{proof} $\mm_\X$-a.e.\ on $E$ we have
\begin{align*}
|\mbox{D}g|&\leq\mbox{lip}(g)&&\text{by \eqref{f_Lip_is_Sob_1}}\\
&=\mbox{lip}({g|}_E)&&\text{by Proposition \ref{prop_lip(f|E)=lip(f)}}\\
&=\mbox{lip}(f\circ\varphi)&\\
&\leq\mbox{Lip}(\varphi)\,\mbox{lip}(f)\circ\varphi&&\text{by \eqref{f_chain_rule_loc_lip_const_1}}\\
&=\mbox{Lip}(\varphi)\,|\mbox{D}f|\circ\varphi &&\text{by \eqref{f_Lip_is_Sob_2}},
\end{align*}
 which is the claim.
\end{proof}
\subsection{Pullback}
In order to define the pullback $\varphi^*$ of a map of bounded deformation $\varphi$, it is first convenient to prove the following auxiliary result: 
\begin{proposition}\label{prop_auxiliary_for_pullback}
Let $(\emph{X},\mathsf{d}_{\emph{X}},\mathfrak{m}_{\emph{X}})$ be a doubling space and $(\emph{Y},\mathsf{d}_{\emph{Y}},\mathfrak{m}_{\emph{Y}})$ be a PI  space. Also, let  $E\in\mathscr{B}(\emph{X})$ and let $\varphi:\,E\to\emph{Y}$ be of bounded deformation.

Then there exists a unique linear and continuous function ${\mathsf{T}}_{\varphi}:\,W^{1,2}(\emph{Y})\to{}{L^{2}(T^{*}\emph{X})|}_{E}$ with the following property: for every $f\in W^{1,2}(\emph{Y})\cap\emph{LIP}(\emph{Y})$, it holds that
\begin{equation}\label{f_auxiliary_for_pullback_1}
\begin{array}{ll}
\nchi_{G}{\mathsf{T}}_{\varphi}(f)=\nchi_{G}\,\emph{d}g\\
\quad
\end{array}\quad\begin{array}{ll}
\mbox{ whenever }G\subseteq{}E\mbox{ is a bounded Borel set and}\\
\mbox{ }g\in{}W^{1,2}(\emph{X})\cap\emph{LIP}(\emph{X})\mbox{ fulfills }{g}\restr G=f\circ{\varphi}\restr G\mbox{.}
\end{array}
\end{equation}
Moreover, the function ${\mathsf{T}}_{\varphi}$ satisfies
\begin{equation}\label{f_auxiliary_for_pullback_2}
{\big|{\mathsf{T}}_{\varphi}(f)\big|}_{*}\leq\emph{Lip}(\varphi)\,|\emph{D}f|\circ\varphi\;\;\;\mathfrak{m}_{\emph{X}}\mbox{-a.e. in }E\mbox{,}\quad\mbox{ for every }f\in{}W^{1,2}(\emph{Y})\mbox{.}
\end{equation}
\end{proposition}
\begin{proof} Write $E=\cup_{i\in\N}E_i$ with the $E_i$'s Borel, bounded and disjoint and let $f\in W^{1,2}(\mbox{Y})\cap {\rm LIP}({\rm Y})$. For every $i\in \N$, let $g_i\in{\rm LIP}({\rm X})\cap W^{1,2}({\rm X})$ be such that $g_i=f\circ\varphi$ on $E_i$ and notice that Proposition \ref{prop_Sobolev_chain_rule_for_Lipschitz_maps} grants that
\[
\big|\sum_i\nchi_{E_i}\d g_i\big|_*=\sum_i\nchi_{E_i}|\d g_i|_*\leq \Lip(\varphi)\sum_i\nchi_{E_i}|\d f|_*\circ\varphi=\Lip(\varphi)\nchi_E|\d f|_*\circ\varphi.
\]
From this inequality it easily follows that the quantity  ${\sf T}_\varphi(f):=\sum_i\nchi_{E_i}\d g_i$ depends only on $f$ and not on the choices of the $E_i$'s and $g_i$'s. With such definition of ${\sf T}_\varphi$ on $W^{1,2}({\Y})\cap{\rm LIP}({\Y})$ we see that the property \eqref{f_auxiliary_for_pullback_1} holds by construction and that  the bound \eqref{f_auxiliary_for_pullback_2} holds for $f\in W^{1,2}(\mbox{Y})\cap {\rm LIP}({\rm Y})$. 

Since $\varphi:E\to {\rm Y}$ is of bounded deformation we also have
\[
\|{\sf T}_\varphi(f)\|_{L^{2}(T^{*}{\X})}\leq\|\Lip(\varphi)\nchi_E|\d f|_*\circ\varphi\|_{L^2({\rm X},\mm_{\rm X})}\leq \Lip(\varphi) \sqrt{\Comp(\varphi)}\|f\|_{W^{1,2}({\rm Y})},
\]
showing that ${\sf T}_\varphi$ is continuous w.r.t.\  the $W^{1,2}({\rm Y})$-norm. Since $\Y$ is doubling, we know (recall \eqref{eq:lipdense}) that  $W^{1,2}({\Y})\cap{\rm LIP}({\Y})$  is dense in $W^{1,2}({\Y})$, hence existence and uniqueness of the continuous extension follow, and by continuity we see that the \eqref{f_auxiliary_for_pullback_2} holds as well.
\end{proof}
By using the fact that $\big\{\mbox{d}f\,:\,f\in W^{1,2}(\mbox{X})\big\}$ generates $L^2(T^*\mbox{X})$ in the sense of modules, one can apply the previous proposition to build the pullback operator:
\begin{theorem}[Pullback of a map of bounded deformation]\label{thm_pullback_bdd_deform_map}
Let $(\emph{X},\mathsf{d}_{\emph{X}},\mathfrak{m}_{\emph{X}})$ be a doubling space and $(\emph{Y},\mathsf{d}_{\emph{Y}},\mathfrak{m}_{\emph{Y}})$ be a PI  space. Also, let  $E\in\mathscr{B}(\emph{X})$ and let $\varphi:\,E\to\emph{Y}$ be of bounded deformation.

Then there exists a unique linear and continuous function $\varphi^{*}:\,L^{2}(T^{*}\emph{Y})\to{}{L^{2}(T^{*}\emph{X})}\restr {E}$, called the \emph{pullback} of $\varphi$, such that
\begin{equation}\label{f_pullback_bdd_deform_map_1}
\begin{array}{ll}\varphi^{*}(\emph{d}f)={\mathsf{T}}_{\varphi}(f)\\
\varphi^{*}(h\omega)=(h\circ\varphi)\varphi^{*}\omega
\end{array}\quad\begin{array}{ll}
\mbox{ for every }f\in{}W^{1,2}(\emph{Y})\cap{\rm LIP}({\rm Y})\mbox{,}\\
\mbox{ for every }\omega\in{}L^{2}(T^{*}\emph{Y})\mbox{ and }h\in{}L^{\infty}(\mathfrak{m}_{\emph{Y}})\mbox{.}
\end{array}
\end{equation}
Moreover, it holds
\begin{equation}\label{f_pullback_bdd_deform_map_2}
{|\varphi^{*}\omega|}_{*}\leq\emph{Lip}(\varphi)\,{|\omega|}_{*}\circ\varphi\;\;\;\mathfrak{m}_{\emph{X}}\mbox{-a.e. in }E\mbox{,}\quad\mbox{ for every }\omega\in{}L^{2}(T^{*}\emph{Y})\mbox{.}
\end{equation}
\end{theorem}
\begin{proof} Start observing that the continuity of ${\sf T}_\varphi$, the density of $W^{1,2}\cap {\rm LIP}(\Y)$ in $W^{1,2}(\Y)$ (recall \eqref{eq:lipdense}), the required continuity of $\varphi^*$ and the first in \eqref{f_pullback_bdd_deform_map_1} force the choice
\begin{equation}
\label{eq:req2}
\varphi^{*}(\d f)={\mathsf{T}}_{\varphi}(f)\qquad\forall f\in W^{1,2}(\Y).
\end{equation}
Now define $V\subset L^2(T^*{\rm Y})$ as
$$V:=\left\{\sum_{i=1}^{n}\chi_{F_{i}}\cdot\mbox{d}f_{i}\;\bigg|
\;\begin{array}{ll}
F_1,\ldots,F_n\in\mathscr{B}(\mbox{Y})\mbox{ are a partition of }\mbox{Y}\\
\mbox{and }f_{1},\ldots,f_{n}\in{}W^{1,2}(\mbox{Y})\mbox{ for some }n\geq{}1
\end{array}
\right\}\mbox{,}$$
and notice that \eqref{eq:req2} and the second in \eqref{f_pullback_bdd_deform_map_1} force the definition
\[
\varphi^*\omega:=\sum_{i=1}^n\nchi_{F_i}\circ\varphi \, {\sf T}_\varphi(f_i)\qquad\text{for}\qquad \omega\in V,\quad \omega=\sum_{i=1}^n\nchi_{F_i}\d f_i
\]
From \eqref{f_auxiliary_for_pullback_2} it directly follows that the bound \eqref{f_pullback_bdd_deform_map_2} holds for such $\omega$'s, showing in particular that such definition of $\varphi^*\omega$ is well posed, in the sense that $\varphi^*\omega$ depends only on $\omega$ and not on the way we write $\omega$ in terms of $F_i$'s and $f_i$'s. It is also clear that the first in \eqref{f_pullback_bdd_deform_map_1} holds and that the second holds for simple functions $h$.

The inequality \eqref{f_pullback_bdd_deform_map_2} and the fact that $\varphi$ is of bounded compression show that $\varphi^*:V\to {L^{2}(T^{*}{\X})}\restr{E}$ is continuous w.r.t.\ the $L^2(T^*{\rm Y})$-norm, and being $V$ dense in $L^2(T^*{\rm Y})$, existence and uniqueness of the continuous extension follow. 

The fact that such extension satisfies \eqref{f_pullback_bdd_deform_map_1} and \eqref{f_pullback_bdd_deform_map_2} follows by simple continuity arguments.
\end{proof}
\begin{theorem}[Functoriality of the pullback]\label{thm_functoriality_pullback}
Let $(\emph{X},\mathsf{d}_{\emph{X}},\mathfrak{m}_{\emph{X}})$ be a doubling space and $(\emph{Y},\mathsf{d}_{\emph{Y}},\mathfrak{m}_{\emph{Y}})$,
$(\emph{Z},\mathsf{d}_{\emph{Z}},\mathfrak{m}_{\emph{Z}})$ be PI spaces. Also, let  $E\in\mathscr{B}(\emph{X})$ and $F\in\mathscr{B}(\emph{Y})$ and  $\varphi:\,E\to\emph{Y}$ and $\psi:\,F\to\emph{Z}$ be two maps of bounded deformation such that $\varphi(E)\subseteq F$. 

Then $\psi\circ\varphi:\,E\to\emph{Z}$ is of bounded deformation and 
\begin{equation}\label{f_functoriality_pullback_1}
(\psi\circ\varphi)^*=\varphi^*\circ \psi^*.
\end{equation}
\end{theorem}
\begin{proof} The fact that $\psi\circ\varphi$ is of bounded deformation is trivial, so that thanks to the characterization of pullback given by Theorem \ref{thm_pullback_bdd_deform_map} and the fact that $\varphi^*\circ \psi^*:L^2(T^*{\rm Z})\to L^2(T^*{\rm X})\restr E$ is linear and continuous,  to conclude it is sufficient to show that $\varphi^*\circ\psi^*$ satisfies the following properties:
\begin{equation}\label{f_functoriality_pullback_2}
\begin{array}{ll}
\varphi^*\big(\psi^*(\mbox{d}f)\big)=\mathsf{T}_{\psi\circ\varphi}(f)\\
\varphi^*\big(\psi^*(h\omega)\big)=(h\circ\psi\circ\varphi)\,\varphi^*(\psi^*\omega)
\end{array}\quad\begin{array}{ll}
\mbox{ for every }f\in W^{1,2}(\mbox{Z})\cap{\rm LIP}({\rm Z})\mbox{,}\\
\mbox{ for every }\omega\in L^2(T^*\mbox{Z})\mbox{ and }h\in L^\infty(\mathfrak{m}_{\mbox{\scriptsize{Z}}})\mbox{.}
\end{array}
\end{equation}
To prove the first equality in \eqref{f_functoriality_pullback_2}, fix $G\in\mathscr{B}(E)$ bounded and $g\in W^{1,2}(\mbox{X})\cap\mbox{LIP}(\mbox{X})$ such that ${g}\restr G={f\circ\psi\circ\varphi}\restr G$. It is enough to show that
\begin{equation}\label{f_functoriality_pullback_3}
\nchi_G\,\varphi^*\big(\psi^*(\mbox{d}f)\big)=\nchi_G\,\mbox{d}g\mbox{.}
\end{equation}
Choose $G'\in\mathscr{B}(F)$ bounded with $\varphi(G)\subseteq G'$ and $g'\in W^{1,2}(\mbox{Y})\cap\mbox{LIP}(\mbox{Y})$ with ${g'}\restr{G'}={f\circ\psi}\restr {G'}$. Thus $\nchi_{G'}\,\psi^*(\mbox{d}f)=\nchi_{G'}\, \mbox{d}g'$ and, since $g=g'\circ\varphi$ on $G$, also $\nchi_G\, \varphi^*(\mbox{d}g')=\nchi_G\,\mbox{d}g$. Hence
$$\nchi_G\, \varphi^*\big(\psi^*(\mbox{d}f)\big)=\nchi_G\, \varphi^*\big(\nchi_{G'}\, \psi^*(\mbox{d}f)\big)
=\nchi_G\, \varphi^*\big(\nchi_{G'}\,\mbox{d}g'\big)=\nchi_G\,\varphi^*(\mbox{d}g')=\nchi_G\,\mbox{d}g\mbox{,}$$
proving \eqref{f_functoriality_pullback_3} and accordingly the first equality in \eqref{f_functoriality_pullback_2}.
To conclude, notice that for $\omega\in L^2(T^*{\rm Z})$ and $h\in L^\infty(\mm_{\rm Z})$ we have
$$\varphi^*\big(\psi^*(h\omega)\big)=\varphi^*\big((h\circ\psi)\,\psi^*\omega\big)=(h\circ\psi\circ\varphi)\,\varphi^*(\psi^*\omega)\mbox{,}$$
hence also the second equality in \eqref{f_functoriality_pullback_2} is satisfied.
\end{proof}
\subsection{Differential}
When $\varphi$ is invertible and also its inverse map is of bounded deformation, it is possible to define the \textsl{differential} of $\varphi$   by  duality with the pullback: 
\begin{proposition}[Differential of a map of bounded deformation]\label{prop_differential_bdd_deform_map}

Let $(\emph{X},\mathsf{d}_{\emph{X}},\mathfrak{m}_{\emph{X}})$ be a doubling space and $(\emph{Y},\mathsf{d}_{\emph{Y}},\mathfrak{m}_{\emph{Y}})$ be a PI  space. Also, let  $E\in\mathscr{B}({\X})$, $F\in \mathscr B(\Y)$ and let $\varphi:\,E\to F$ be an invertible function such that both $\varphi$ and $\varphi^{-1}$ are of bounded deformation.

Then there exists a unique map $\emph{d}\varphi:\,{L^2(T\emph{X})}\restr E\to{L^2(T\emph{Y})}\restr F$ such that for any $\mathsf{v}\in{L^2(T\emph{X})}\restr E$, one has
\begin{equation}\label{f_differential_bdd_deform_map_1}
\omega\big(\emph{d}\varphi(\mathsf{v})\big)=(\varphi^*\omega)(\mathsf{v})\circ\varphi^{-1}\;\;\mathfrak{m}_{\emph{Y}}\mbox{-a.e. in }F\mbox{,}\quad\mbox{ for every }\omega\in{L^2(T^*\emph{Y})}\restr F\mbox{.}
\end{equation}
The operator $\emph{d}\varphi$, called \emph{differential} of $\varphi$, is linear continuous and satisfies
\begin{equation}\label{f_differential_bdd_deform_map_2}
\big|\emph{d}\varphi(\mathsf{v})\big|\leq\emph{Lip}(\varphi)\,|\mathsf{v}|\circ\varphi^{-1}\;\;\mathfrak{m}_{\emph{Y}}\mbox{-a.e. in }F\mbox{,}\quad\mbox{ for every }\mathsf{v}\in{L^2(T\emph{X})}\restr E\mbox{.}
\end{equation}
Moreover, if $E'\subset E$ is Borel and $F':=\varphi(E')$, letting $\d\varphi\restr{E'}$ be the differential of $\varphi$ seen as a map from $E'$ to $F'$, we have that
\begin{equation}
\label{eq:locdiff}
\d({\varphi}\restr {E'})(\nchi_{E'}\mathsf{v})=\nchi_{F'}{\d}\varphi(\mathsf{v})\quad\mbox{ for every }\mathsf{v}\in {L^2(T\mbox{\X})}\restr E\mbox{.}
\end{equation}
\end{proposition}
\begin{proof} For ${\sf v}\in L^2(T{\rm X})\restr E$ and $\omega\in L^2(T^*\Y)$ we have
\begin{equation}
\label{eq:aereo1}
|\varphi^*\omega({\sf v})\circ\varphi^{-1}|\leq |\varphi^*\omega|_*\circ\varphi^{-1}|{\sf v}|\circ\varphi^{-1}\stackrel{\eqref{f_pullback_bdd_deform_map_2}}\leq \Lip(\varphi)|\omega|_*|{\sf v}|\circ\varphi^{-1}\qquad\mm_{\rm Y}-a.e.\ on\ F.
\end{equation}
Taking into account the fact that  $\varphi^{-1}$ is of bounded compression we see that $|{\sf v}|\circ\varphi^{-1}\in L^2(F)$ hence the above inequality grants that the function $L_{\sf v}(\omega):= \nchi_F\varphi^*\omega({\sf v})\circ\varphi^{-1}$, intended to be 0 outside $F$, belongs to $L^1(\Y)$. The same bound also grants that the linear map $L^2(T^*{\rm Y})\ni \omega\mapsto L_{\sf v}(\omega)\in L^1(\Y)$ is continuous and since for  $h\in L^\infty(\mm_{\rm Y})$ we  have
\[
L_{\sf v}(h\omega)=\nchi_F\varphi^*(h\omega)({\sf v})\circ\varphi^{-1}=\nchi_F\big(h\circ\varphi\,\varphi^*\omega({\sf v})\big)\circ\varphi^{-1}=h\nchi_F\big(\varphi^*\omega({\sf v})\big)\circ\varphi^{-1}=hL_{\sf v}(\omega),
\]
we see, by the very definition of tangent module, that $L_{\sf v}$ is an element of $L^2(T\Y)$ which from now on we shall denote by $\d\varphi({\sf v})$. 

The fact that  the map ${\sf v}\mapsto\d\varphi({\sf v})$ is linear is clear from the definition and the bound \eqref{f_differential_bdd_deform_map_2} is a restatement of \eqref{eq:aereo1}. The inequality \eqref{f_differential_bdd_deform_map_2} also grants the continuity of $L^2(T\X)\restr E\ni {\sf v}\mapsto \d\varphi({\sf v})\in L^2(T\Y)$, indeed:
\[
\|\d\varphi({\sf v})\|_{L^2(T\Y)}\stackrel{\eqref{f_differential_bdd_deform_map_2}}{\leq} \Lip(\varphi)\||{\sf v}\circ\varphi^{-1}|\|_{L^2(F)}\leq\Lip(\varphi)\sqrt{\Comp(\varphi^{-1})}\|{\sf v}\|_{L^2(T\X)\restr E}.
\]
Finally, the last claim \eqref{eq:locdiff}  is a direct consequence of the characterizing property \eqref{f_differential_bdd_deform_map_1}.
\end{proof}
A direct consequence of Theorem \ref{thm_functoriality_pullback} is the chain rule for the  differential:
\begin{corollary}[Chain rule for the differential]\label{cor_functoriality_differential} Let  $(\emph{X},\mathsf{d}_{\emph{X}},\mathfrak{m}_{\emph{X}})$ be a doubling space and  $(\emph{Y},\mathsf{d}_{\emph{Y}},\mathfrak{m}_{\emph{Y}})$, $(\emph{Z},\mathsf{d}_{\emph{Z}},\mathfrak{m}_{\emph{Z}})$ be PI spaces.
Let $E\in\mathscr{B}(\emph{X})$, $F\in\mathscr{B}(\emph{Y})$ and $G\in\mathscr{B}(\emph{Z})$, and $\varphi:\,E\to F$ and $\psi:\,F\to G$ be both invertible, of bounded deformation  and with inverses  also of bounded deformation. Then
\begin{equation}\label{f_functoriality_differential}
\emph{d}(\psi\circ\varphi)(\mathsf{v})=\emph{d}\psi\big(\emph{d}\varphi(\mathsf{v})\big)\quad\mbox{ for every }\mathsf{v}\in
{L^2(T\emph{X})}\restr E\mbox{.}
\end{equation}
\end{corollary}
\begin{proof} Just notice that for   $\mathsf{v}\in{L^2(T\mbox{X})}\restr E$ and $\omega\in{L^2(T^*\mbox{Z})}\restr G$ we have
\[\begin{split}
\omega\big(\mbox{d}(\psi\circ\varphi)(\mathsf{v})\big)&=\big((\psi\circ\varphi)^*\omega\big)(\mathsf{v})\circ{(\psi\circ\varphi)}^{-1}\stackrel{\eqref{f_functoriality_pullback_1}}=\big(\varphi^*(\psi^*\omega)\big)(\mathsf{v})\circ\varphi^{-1}\circ\psi^{-1}\\
&=(\psi^*\omega)\big(\mbox{d}\varphi(\mathsf{v})\big)\circ\psi^{-1}=\omega\big(\mbox{d}\psi\big(\mbox{d}\varphi(\mathsf{v})\big)\big)
\end{split}\]
 $\mathfrak{m}_{\mbox{\scriptsize{Z}}}$-a.e. in $G$. By Theorem \ref{prop_differential_bdd_deform_map} this is sufficient to conclude.
\end{proof}
Let us now discuss the case $\X=\Y=\R^k$. Let $E,F\subset \R^k$ Borel and $\varphi:E\to F$ of bounded deformation, invertible and with inverse of bounded deformation. Then in particular it is Lipschitz and we know from Rademacher theorem that for $\mathcal L^k$-a.e.\ $x\in E$ its differential, intended in the classical sense, exists. Such differential, which for the moment we denote by $\hat \d\varphi(x)$, is a linear map from $T_x\R^k\sim\R^k$ to $T_{\varphi(x)}\R^k\sim\R^k$. Thus we can see   $\hat\d\varphi$ as a map from $L^2(E,\R^k)$ to $L^2(F,\R^k)$. Recalling that we identified $L^2(E,\R^k)$ with the restriction $L^2(T\R^k)\restr E$ of the tangent module $L^2(T\R^k)$ to $E$ (and similarly for $F$), from the characterization of the `abstract' differential $\d\varphi$ provided by Proposition \ref{prop_differential_bdd_deform_map} it is readily verified that $\d\varphi$ and $\hat\d\varphi$ can be canonically identified, in the sense that for ${\sf v}\in L^2(T\R^k)\restr E\sim L^2(E,\R^k)$ it holds
\begin{equation}\label{f_differentials_in_R^k}
\big({\rm d}\varphi(\mathsf{v})\big)\big(\varphi(x)\big)=\hat{\rm{d}}\varphi(x)\big(\mathsf{v}(x)\big)\quad\mbox{ for }\mathcal{L}^k\mbox{-a.e. }x\in E.
\end{equation}
Due to this fact, from now on we won't distinguish anymore between the objects $\hat\d\varphi$ and $\d\varphi$.
\section{Strongly $\mm$-Rectifiable Spaces}
We introduce a new class of metric measure spaces, called strongly $\mathfrak{m}$-rectifiable spaces. Roughly speaking, these spaces can be partitioned (up to negligible sets) into countably many Borel sets, which are biLipschitz equivalent to suitable subsets of the Euclidean space, by means of maps that also keep under control the measure.

For the sake of simplicity, it is convenient to use the following notation: given a measured space $(S,\mathcal{M},\mu)$, we say that $(E_i)_{i\in\mathbb{N}}\subseteq\mathcal{M}$ is a \textsl{$\mu$-partition} of $E\in\mathcal{M}$ provided it is a partition of some $F\in\mathcal{M}$ such that $F\subseteq E$ and $\mu(E\setminus F)=0$. Moreover, given two $\mu$-partitions $(E_i)_i$ and $(F_j)_j$ of $E$, we say that $(F_j)_j$ is a \textsl{refinement} of $(E_i)_i$ if for every $j\in\mathbb{N}$ with $F_j\neq\emptyset$ there exists (a unique) $i\in\mathbb{N}$ such that $F_j\subseteq E_i$.
\begin{definition}[Strongly $\mathfrak{m}$-rectifiable space]\label{def_strongly_rectifiable_space}
A metric measure space $(\emph{X},\mathsf{d},\mathfrak{m})$ is said to be $\mathfrak{m}$\emph{-rectifiable}   provided it is a disjoint union $\emph{X}=\cup_{k\in\N}A_k$ of suitable $(A_k)\subset \mathscr{B}(\emph{X})$, such that the following condition is satisfied:
given $k\in\N$, there exists an $\mathfrak{m}$-partition $(U_i)_{i\in\mathbb{N}}\subseteq\mathscr{B}(\emph{X})$ of $A_k$ and a sequence $(\varphi_i)_{i\in\mathbb{N}}$ of maps $\varphi_i:\,U_i\to\mathbb{R}^k$ such that
\begin{equation}\label{f_strongly_rectifiable_space}
\varphi_i:\,U_i\to\varphi_i(U_i)\mbox{ is biLipschitz and }
(\varphi_i)_*(\mathfrak{m}\restr{U_i})\ll\mathcal{L}^k\mbox{ for every }i\in\mathbb{N}\mbox{.}
\end{equation}
The partition $\X=\cup_{k\in\N}A_k$ - which is clearly unique up to modification of negligible sets - is called \emph{dimensional decomposition} of $\X$.

$(\emph{X},\mathsf{d},\mathfrak{m})$ is said to be \emph{strongly} $\mathfrak{m}$\emph{-rectifiable} provided  for every $\eps>0$ the $(U_i,\varphi_i)$   can be chosen so that the $\varphi_i$ are  $(1+\varepsilon)$-biLipschitz.
\end{definition}
When working on $\mathfrak{m}$-rectifiable spaces, it is natural to adopt the following terminology, which is inspired by the language of differential geometry:
\begin{definition}[Charts and atlases]\label{def_charts_and_atlases}
Let $(\emph{X},\mathsf{d},\mathfrak{m})$ be an $\mathfrak{m}$-rectifiable metric measure space. A \emph{chart} on $X$ is a couple $(U,\varphi)$ 
where $U\in\mathscr{B}(A_k)$ for some $k\in\N $ and $\varphi:\,U\to\mathbb{R}^k$ satisfies
\begin{equation}\label{f_chart}
\begin{split}
&\varphi:\,U\to\varphi(U)\quad\mbox{ is biLipschitz,}\\
&C^{-1}\,\mathcal{L}^k\restr{\varphi(U)}
\leq\varphi_*(\mathfrak{m}\llcorner U)
\leq C\,\mathcal{L}^k\restr{\varphi(U)}\mbox{,}\end{split}
\end{equation}
for a suitable constant $C\geq 1$. An \emph{atlas} on $(\emph{X},\mathsf{d},\mathfrak{m})$ is a family $\mathscr{A}=\bigcup_{k\in\N}\big\{(U^k_i,\varphi^k_i)\big\}_{i\in\mathbb{N}}$ of charts on $(\emph{X},\mathsf{d},\mathfrak{m})$ such that $(U^k_i)_{i\in\mathbb{N}}$ is an $\mathfrak{m}$-partition of $A_k$ for every $k\in\N$.

The chart $(U,\varphi)$ is said to be an $\eps${\rm-chart} provided $\varphi:U\to\varphi(U)$ is $(1+\eps)$-biLipschitz and an atlas is said to be an $\eps${\rm-atlas} provided all of its charts are $\eps$-charts.
\end{definition}
We collect few simple facts about atlases which we shall frequently use in what follows:
\begin{itemize}
\item[i)] A $\mm$-rectifiable space admits an atlas and a strongly $\mm$-rectifiable space admits an $\eps$-atlas for every $\eps>0$. Indeed, for $(U_i,\varphi_i)$ as in \eqref{f_strongly_rectifiable_space} we can consider the density $\rho_i$ of $\varphi_*\mm\restr{U_i}$ w.r.t.\ the Lebesgue measure and the sets $U_{i,j}:=\varphi_i^{-1}(\{\rho_i\in[2^j,2^{j+1})\})$, $j\in\mathbb Z$. It is clear that $(U_{i,j},\varphi_i\restr{U_{i,j}})$ is a chart for every $j$ and that the $U_{i,j}$'s provide a $\mm$-partition of $U_i$, so that repeating the construction for every $i$ yields the desired atlas.
\item[ii)] Let  $(U_i,\varphi_i)_{i\in\N}$ be an atlas and, for every $i$, $(U_{i,j})_{j\in\N}$ a $\mm$-partition of $U_i$. Then $(U_{i,j},\varphi_i\restr{U_{i,j}})_{i,j\in\N}$ is also an atlas. In particular, by inner regularity of $\mm$, every $\mm$-rectifiable space admits an atlas whose charts are defined on compact sets.
\end{itemize}%
A first property of $\mathfrak{m}$-rectifiable spaces, whose proof is based upon the notion of differential
introduced in Proposition \ref{prop_differential_bdd_deform_map}, is the following:
\begin{theorem}[Dimensional decomposition of the tangent module]\label{thm_dim_decompos_of_L^2(TX)}
Let $(\rm{X},\mathsf{d},\mathfrak{m})$ be a PI space which is also $\mm$-rectifiable  and let  $(A_k)$ be its dimensional decomposition.

Then for every $k\in\N$ such that $\mm(A_k)>0$ we have that $L^2(T\rm{X})$ has dimension $k$ on $A_k$.
\end{theorem}
\begin{proof}
Let $\mathscr{A}=\big\{(U^k_i,\varphi^k_i)\big\}_{k,i}$ be an atlas on $\rm{X}$. The claim is equivalent to the fact that for every $U^k_i$ with $\mm(U^k_i)>0$ the dimension of $L^2(T\X)$ on $U^k_i$ is $k$. For such $U^k_i$ the map $\varphi^k_i$ is invertible, of bounded deformation and with inverse of bounded deformation, hence taking into account Corollary \ref{cor_functoriality_differential} we see that $\d\varphi^k_i:L^2(T\X)\restr{U^k_i}\to L^2(T\R^k)\restr{\varphi^k_i(U^k_i)}$ is continuous, invertible, with continuous inverse and sends $h{\sf v}$ to $h\circ (\varphi^k_i)^{-1}\d\varphi^k_i({\sf v})$. It is then clear that the dimensions of $L^2(T\X)\restr{U^k_i}$ and $L^2(T\R^k)\restr{\varphi^k_i(U^k_i)}$ coincide and, since the latter has dimension $k$, the conclusion follows.
\end{proof}
\begin{remark}{\rm
Using the finite dimensionality results obtained by Cheeger in \cite{Cheeger00} it is not hard to see that the dimensional decomposition $(A_k)$ of a PI space which is also $\mm$-rectifiable must be so that $\mm(A_k)=0$ for all $k$ sufficiently large. Yet, our discussion is independent on this specific result and thus we won't insist on this point.
}\fr\end{remark}
When we restrict our attention to the smaller class of strongly $\mathfrak{m}$-rectifiable spaces, we have a stronger geometric characterization of the tangent module. Section \ref{sect_equiv_L^2(TX)_and_L^2(T_GH_X)} will be entirely devoted to describe such result.
In order to further develop our theory in that direction, we need to provide any strongly $\mathfrak{m}$-rectifiable space $(\mbox{X},\mathsf{d},\mathfrak{m})$ with a special sequence of atlases, which are aligned in a suitable sense.

\begin{definition}[Aligned family of atlases]\label{def_aligned_atlases}
Let $(\emph{X},\mathsf{d},\mathfrak{m})$ be a strongly $\mathfrak{m}$-rectifiable space.
Let $\varepsilon_n\downarrow 0$ and $\delta_n\downarrow 0$. Let $(\mathscr{A}_n)_{n\in\mathbb{N}}$ be a sequence of atlases on $\emph{X}$. Then we say that $(\mathscr{A}_n)_n$ is an \emph{aligned family of atlases} of parameters $\varepsilon_n$ and $\delta_n$ provided the following conditions are satisfied:
\begin{itemize}
\item[\emph{(i)}] Each $\mathscr{A}_n=\big\{(U^{k,n}_i,\varphi^{k,n}_i)\big\}_{k,i}$ is an $\varepsilon_n$-atlas and the domains $U^{k,n}_i$ are compact.
\item[\emph{(ii)}] The family $(U^{k,n}_i)_{k,i}$ is a refinement of $(U^{k,n-1}_j)_{k,j}$ for any $n\in\mathbb{N}^+$.
\item[\emph{(iii)}] If $n\in\mathbb{N}^+$, $k\in\N$ and $i,j\in\mathbb{N}$ satisfy $U^{k,n}_i\subseteq U^{k,n-1}_j$, then
\begin{equation}\label{f_aligned_atlases}
\bigg\|\,{\emph{d}}\Big(\emph{id}_{\mathbb{R}^k}-\varphi^{k,n-1}_j\circ\big(\varphi^{k,n}_i\big)^{-1}\Big)(y)\bigg\|\leq\delta_n\quad\mbox{ for }\mathcal{L}^k\mbox{-a.e. }y\in\varphi^{k,n}_i(U^{k,n}_i)\mbox{.}
\end{equation}
\end{itemize}
\end{definition}
The discussions made before grant that any strongly $\mm$-rectifiable space admits atlases satisfying $(i),(ii)$ above. In fact, as we shall see in a moment, also $(iii)$ can be fulfilled by an appropriate choice of atlases, but in order to show this we need a small digression.

Recall that $O(\R^k)$ denotes the group of linear isometries of $\R^k$ and for $\eps>0$ let us introduce
\[
O^\eps(\R^k):=\Big\{T:\R^k\to\R^k\text{ linear, invertible and such that }\|T\|,\|T^{-1}\|\leq 1+\eps\Big\}.
\]
Notice that $O^\eps(\R^k)$ - being closed and bounded - is compact for every $\eps>0$ and that $O(\R^k)=\cap_{\eps>0}O^\eps(\R^k)$. Then we have the following simple result:
\begin{proposition}\label{lemma_approximation_R_Borel}
Let $k\in\N$ and $\delta>0$. Then there exists $\eps>0$ and a Borel map $R:O^\eps(\R^k)\to O(\R^k)$ with finite image such that
\[
\|T-R(T)\|\leq\delta\qquad\forall T\in  O^\eps(\R^k).
\]
\end{proposition}
\begin{proof} From the compactness of $O(\R^k)$ we know that there are $T_1,\ldots,T_n\in O(\R^k)$ such that $O(\R^k)\subset U_\delta:=\cup_iB_\delta(T_i)$. We claim that there exists $\eps>0$ such that $O^\eps(\R^k)\subset U_\delta$ and argue by contradiction. If not, the compact set $K^\eps:=O^\eps(\R^k)\setminus U_\delta$ would be not empty for every $\eps>0$. Since clearly $K^\eps\subset K^{\eps'}$ for $\eps\leq\eps'$, the family $K^\eps$ has the finite intersection property, but on the other hand the identity $O(\R^k)=\cap_{\eps>0}O^\eps(\R^k)$ yields $\cap_{\eps>0}K^\eps=\emptyset$, which is a contradiction.

Thus there is $\eps>0$ such that $O^\eps(\R^k)\subset U_\delta$. For such $\eps$ we define $R:O^\eps(\R^k)\to O(\R^k)$ to be equal to $T_1$ on $B_\delta(T_1)$ and then recursively equal to $T_n$ on $B_\delta(T_n)\setminus\cup_{i<n}B_\delta(T_i)$.
\end{proof}
Using Proposition \ref{lemma_approximation_R_Borel} it is possible to show that any strongly $\mathfrak{m}$-rectifiable space admits an aligned family of atlases:
\begin{theorem}\label{thm_existence_aligned_atlases}
Let $(\emph{X},\mathsf{d},\mathfrak{m})$ be a strongly $\mathfrak{m}$-rectifiable metric measure space and  $\varepsilon_n\downarrow 0$ and $\delta_n\downarrow 0$ be two given sequences. Then $\emph{X}$ admits an aligned family $(\mathscr{A}_n)_n$ of atlases of parameters $\varepsilon_n$ and $\delta_n$.
\end{theorem}
\begin{proof} Let $(A_k)$ be the dimensional decomposition of $\X$ and notice that to conclude it is sufficient to build, for every $k\in\N$, aligned charts as in $(iii)$ of Definition \ref{def_aligned_atlases} covering $\mm$-almost all $A_k$.

For  $k,n\in\mathbb{N}$, let $\varepsilon'_{n,k}$ be associated to $\delta_n$ and $k$ as in Proposition \ref{lemma_approximation_R_Borel} and choose $\bar{\varepsilon}_{n,k}>0$ such that
\begin{equation}\label{f_existence_aligned_atlases_1}
\bar{\varepsilon}_{n,k}\leq\varepsilon_n\quad\mbox{ and }\quad(1+\bar{\varepsilon}_{n-1,k})(1+\bar{\varepsilon}_{n,k})\leq1+\varepsilon'_{n,k}\quad\mbox{ for every }k,n\in\mathbb{N}\mbox{.}
\end{equation}
We now construct  the required aligned family $(\mathscr{A}_n)_n$ of atlases   by  recursion: start observing that since $(\X,\sfd,\mm)$ is strongly $\mm$-rectifiable, there exists an atlas $\mathscr{A}_0$ such that the charts with domain included in $A_k$ are  $\bar\eps_{0,k}$-biLipschitz. Now assume that for some $n\in\N$  we have already defined $\mathscr{A}_0,\ldots,\mathscr{A}_{n-1}$ satisfying the alignment conditions and say that $\mathscr{A}_{n-1}=\big\{(U^k_i,\varphi^k_i)\big\}_{k,i}$. Again using the strong $\mm$-rectifiability of $\X$, find an atlas   $\big\{(V^k_j,\psi^k_j)\big\}_{k,j}$   whose domains $(V^k_j)_{k,j}$ constitute a refinement of the domains $(U^k_i)_{k,i}$ of $\mathscr{A}_{n-1}$ and such that those charts with domain included in $A_k$ are $\bar\eps_{n,k}$-biLipschitz. 

Fix  $k,j\in\mathbb{N}$ and let $i\in\mathbb{N}$ be the unique index  such that $V^k_j\subseteq U^k_i$. For the sake of brevity, let us denote by $\tau$ the transition map $\varphi^k_i\circ(\psi^k_j)^{-1}:\,\psi^k_j(V^k_j)\to\varphi^k_i(V^k_j)$ and observe that it is $(1+\varepsilon'_{n,k})$-biLipschitz by \eqref{f_existence_aligned_atlases_1}. Hence its  differential $\d\tau$ satisfies $\big\|{\mbox{d}}\tau(y)\big\|,\big\|{\mbox{d}}\tau(y)^{-1}\big\|\leq 1+\varepsilon'_{n,k}$, or equivalently 
${\mbox{d}}\tau(y)\in O^{\eps'_{n,k}}(\R^k)$, for $\mathcal{L}^k$-a.e. $y\in\psi^k_j(V^k_j)$.

Let $R:\,O^{\eps'_{n,k}}(\R^k)\to{O}(\mathbb{R}^k)$ be given by Proposition \ref{lemma_approximation_R_Borel} with  $\delta:=\delta_n$ and denote by  $F^k_j\subset O(\R^k)$ its finite image. For $T\in F^k_j$ let $P_T:=(R\circ\d\tau)^{-1}(T)\subset \R^k$, so that $(P_T)_{T\in F^k_j}$ is a $\mathcal L^k$-partition of $\psi^k_j(V^k_j)$. 

For $\mathcal L^k$-a.e.\ $y\in T(P_T)\subset\R^k$ we have
\begin{equation}
\label{eq:allineato}
\begin{split}
\Big\|{\mbox{d}}\big(\varphi^k_i\circ\big(T\circ\psi^k_j\big)^{-1}-\mbox{id}_{\mathbb{R}^k}\big)(y)\Big\|
&=\Big\|{\mbox{d}}\big(\tau\circ T^{-1}-\mbox{id}_{\mathbb{R}^k}\big)(y)\Big\|\\
&=\Big\|{\mbox{d}}\big((\tau-T)\circ T^{-1}\big)(y)\Big\|\\
&\leq\big\|{\mbox{d}}\tau(T^{-1}(y))-T\,\big\|\,\big\|T^{-1}\big\|\\
&=\big\|{\mbox{d}}\tau(T^{-1}(y))-T\big\|\\
(\text{because $T^{-1}(y)\in P_T$})\qquad&=\big\|{\mbox{d}}\tau(T^{-1}(y))-R\big({\mbox{d}}\tau(T^{-1}(y))\big)\big\|\\
(\text{by definition of $R$})\qquad&\leq\delta_n.
\end{split}
\end{equation}
We therefore define 
\begin{equation}\label{f_existence_aligned_atlases_2}
\bar{U}^k_{j,T}:=(\psi^k_j)^{-1}(P_T)\quad\mbox{ and }\quad
\bar{\varphi}^k_{j,T}:=T\circ\psi^k_j\restr{\bar{U}^k_{j,T}}\quad\mbox{ for every }T\in F^k_j\mbox{,}
\end{equation}
so that accordingly
\begin{equation}\label{f_existence_aligned_atlases_3}
\mathscr{A}_n:=
\Big\{\big(\bar{U}^k_{j,T},\bar{\varphi}^k_{j,T}\big)\,:\,k,j\in\mathbb{N},\,T\in F^k_j\Big\}
\end{equation}
is an atlas on $(\mbox{X},\mathsf{d},\mathfrak{m})$, which fulfills (ii), (iii) of Definition \ref{def_aligned_atlases} and such that the charts with domain included in $A_k$ are $\bar\eps_{n,k}$-biLipschitz.

Up to a further refining we can assume that the charts in $\mathscr{A}_n$ have compact domains and, since $\bar\eps_{n,k}\leq \eps_n$ for every $k,n\in\N$, the thesis is proved.
\end{proof}
\section{Gromov-Hausdorff Tangent Module}
\subsection{Measurable Banach Bundle}
Let $(\mbox{X},\mathsf{d},\mathfrak{m})$ be a fixed metric measure space. We propose a notion of measurable Banach bundle:
%$(\mbox{X},\mathsf{d},\mathfrak{m})$:
\begin{definition}[Measurable Banach bundle]\label{def_meas_Banach_bundle}
The quadruplet $\mathbb{T}:=\big(T,\mathcal{M},\pi,\boldsymbol{\mathsf{n}}\big)$ is said to be a \emph{measurable Banach bundle} over $(\emph{X},\mathsf{d},\mathfrak{m})$ provided:
\begin{itemize}
\item[i)] $\mathcal M$ is a $\sigma$-algebra over the set $T$,
\item[ii)] $\pi$ is a measurable map  from $(T,\mathcal{M})$ to $\big(\mbox{\X},\mathscr{B}(\mbox{\X})\big)$ which we shall call  {\rm projection} and
\begin{equation}\label{f_projection_mvb}
T_x:=\pi^{-1}\big(\{x\}\big)\quad\mbox{ is an }\mathbb{R}\mbox{-vector space for }\mathfrak{m}\mbox{-a.e. }x\in\X\mbox{.}
\end{equation}
\item[iii)] $\boldsymbol{\mathsf{n}}:\,T\to[0,+\infty)$ is  a measurable map which we shall call {\rm norm}  such that for $\mm$-a.e.\ $x\in X$ it holds:
\begin{equation}\label{f_norm_mvb}
\begin{split}
&{\boldsymbol{\mathsf{n}}|}_{T_x}\mbox{ is a norm on }T_x\mbox{,}\\
&\big(T_x,{\boldsymbol{\mathsf{n}}}\restr{T_x}\big)\mbox{ is a Banach space,}\\
&\mathscr{B}(T_x)=\mathcal{M}\restr{T_x}:=\big\{E\cap T_x\,:\,E\in\mathcal{M}\big\}\mbox{.}
\end{split}
\end{equation}
\end{itemize}
\end{definition}

Given two measurable Banach bundles $\mathbb{T}_i=\big(T_i,\mathcal{M}_i,\pi_i,\boldsymbol{\mathsf{n}}_i\big)$, $i=1,2$, a \textsl{bundle morphism} is a measurable map $\varphi:\,T_1\to T_2$ such that for $\mm$-a.e.\ $x\in \X$ it holds
\begin{equation}\label{f_bundle_morphism}
\begin{split}
&\varphi\mbox{ maps }(T_1)_x\mbox{ into }(T_2)_x\mbox{,}\\
&{\varphi}\restr{(T_1)_x}\mbox{ is linear and }1\mbox{-Lipschitz from }\big((T_1)_x,{\boldsymbol{\mathsf{n}}_1}\restr{(T_1)_x}\big)\mbox{ to }\big((T_2)_x,{\boldsymbol{\mathsf{n}}_2}\restr{(T_2)_x}\big)\mbox{.}
\end{split}
\end{equation}
Two bundle morphisms $\varphi,\psi:\,T_1\to T_2$ are declared to be equivalent provided
\begin{equation}\label{f_equiv_rel_bundle_morphisms}
{\varphi}\restr{(T_1)_x}={\psi}\restr{(T_1)_x}\mbox{ for }\mathfrak{m}\mbox{-a.e. }x\in\mbox{X}\mbox{}
\end{equation}
and accordingly two measurable Banach bundles $\mathbb{T}_i=\big(T_i,\mathcal{M}_i,\pi_i,\boldsymbol{\mathsf{n}}_i\big)$, $i=1,2$ are declared to be isomorphic provided there are bundle morphisms $\varphi: T_1\to T_2$ and $\psi: T_2\to T_1$ such that $\varphi\circ\psi\sim{\rm id}_{T_2}$ and $\psi\circ\varphi\sim{\rm id}_{T_1}$, which is the same as to say that
\[\begin{split}
&{\psi\circ\varphi}\restr{(T_1)_x}=\mbox{id}_{(T_1)_x}\mbox{ and }{\varphi\circ\psi}\restr{(T_2)_x}=\mbox{id}_{(T_2)_x}
\mbox{ for }\mathfrak{m}\mbox{-a.e. }x\in\mbox{X}\mbox{,}\\
&{\varphi}\restr{(T_1)_x}:\,(T_1)_x\to
(T_2)_x\mbox{ is an isometric isomorphism for }
\mathfrak{m}\mbox{-a.e. }x\in\mbox{X}\mbox{.}
\end{split}\]
Let $\mathbb{T}=\big(T,\mathcal{M},\pi,\boldsymbol{\mathsf{n}}\big)$ be a measurable Banach bundle over $\mathbb{X}$. A \textsl{ section} of $\mathbb{T}$ is a measurable function $\mathsf{v}:\,\mbox{X}\to T$ such that $\pi\circ\mathsf{v}=\mbox{id}_{\mbox{\scriptsize{X}}}$ holds $\mathfrak{m}$-a.e. in $\mbox{X}$. We denote by $[\mathsf{v}]$ the equivalence class of $\mathsf{v}$ with respect to $\mathfrak{m}$-a.e.\ equality and introduce the space $L^2(\mathbb{T})$ of $L^2$-sections as
\begin{equation}\label{f_L^2(T)}
L^2(\mathbb{T}):=\left\{[\mathsf{v}]\;\bigg|\;\mathsf{v}\mbox{ is a section of }\mathbb{T}\mbox{ with }\int_{\mbox{\scriptsize{X}}}\boldsymbol{\mathsf{n}}\big(\mathsf{v}(x)\big)^2\,\mbox{d}\mathfrak{m}(x)<+\infty\right\}\mbox{.}
\end{equation}
With a (common) slight abuse of notation, the elements of $L^2(\mathbb{T})$ will be denoted by $\mathsf{v}$ instead of $[\mathsf{v}]$.

Notice that $L^2(\mathbb{T})$ has a canonical structure of  $L^2(\mathfrak{m})$-normed $L^\infty(\mathfrak{m})$-module on ${\rm X}$: for  $\mathsf{v},\mathsf{w}\in L^2(\mathbb{T})$, $\alpha,\beta\in\mathbb{R}$ and $h\in L^\infty(\mathfrak{m})$  define
\begin{equation}\label{f_ptwse_operations_L^2(T)}
\begin{split}
&(\alpha\,\mathsf{v}+\beta\,\mathsf{w})(x):=\alpha\,\mathsf{v}(x)+\beta\,\mathsf{w}(x)\in T_x\mbox{,}\\
&(h\mathsf{v})(x):=h(x)\,\mathsf{v}(x)\in T_x\mbox{,}\\
&|\mathsf{v}|(x):=\boldsymbol{\mathsf{n}}\big(\mathsf{v}(x)\big)\mbox{,}
\end{split}
\end{equation}
for $\mathfrak{m}$-a.e. $x\in\mbox{X}$.
\begin{remark}\label{rmk_T_to_L^2(T)_functor}{\rm
The collection of measurable Banach bundles on $\X$ and of isomorphism classes of bundle morphisms form a category, which we shall denote by $\boldsymbol{\rm{MBB}}({\X})$.

Similarly, the collection of $L^2(\mathfrak{m})$-normed $L^\infty(\mathfrak{m})$-modules on ${\X}$ and of 1-Lipschitz module morphisms between them form a category, which we denote by   $\boldsymbol{\rm{Mod}}_{2-L^\infty}({\X})$.

The map which sends  each measurable Banach bundle $\mathbb{T}$ to the space of its $L^2$-sections $L^2(\mathbb{T})$ and  each bundle morphism $\varphi:\,T_1\to T_2$ to the map $L^2(\mathbb{T}_1)\ni\mathsf{v}\mapsto\varphi\circ\mathsf{v}\in L^2(\mathbb{T}_2)$, is easily seen to be a fully faithful functor, so that $\boldsymbol{\rm{MBB}}({\X})$ can be thought of as a full subcategory of $\boldsymbol{\rm{Mod}}_{2-L^\infty}({\X})$.
\fr}\end{remark}
\subsection{Gromov-Hausdorff Tangent Bundle}\label{subsect_GH_Tangent_Bundle}

Recall that given a measurable space $(S,\mathcal{M})$, a set $S'$ and a function $f:\,S\to S'$,  the \textsl{push-forward} $f_*\mathcal M$ of $\mathcal{M}$ via $f$ is the $\sigma$-algebra on $S'$ defined by
\begin{equation}\label{f_sigma-algebra_push_forward}
f_*\mathcal{M}:=\big\{E\subseteq S'\,:\,f^{-1}(E)\in\mathcal{M}\big\}\mbox{.}
\end{equation}
Notice that $f_*\mathcal{M}$ is the greatest $\sigma$-algebra $\mathcal{M}'$ on $S'$ for which the function $f$ is measurable from $(S,\mathcal{M})$ to $(S',\mathcal{M}')$.

With this said, let $(\X,\sfd,\mm)$ be a strongly $\mm$-rectifiable metric measure space,  $(A_k)$ its dimensional decomposition and define the following objects:
\begin{itemize}
\item[i)] The set $T_{\rm{GH}}\mbox{X}$ is defined as
\begin{equation}\label{f_T_GH_X}
T_{\rm{GH}}\mbox{X}:=\bigsqcup_{k\in\N} A_k\times\mathbb{R}^k
\end{equation}
and the $\sigma$-algebra $\mathcal{M}_{\rm{GH}}(\mbox{X})$ is given by
\begin{equation}\label{f_M_GH(X)}
\mathcal{M}_{\rm{GH}}(\mbox{X}):=\bigcap_{k\in\N}(\iota_k)_*\mathscr{B}(A_k\times\mathbb{R}^k)\mbox{,}
\end{equation}
where  $\iota_k:\,A_k\times\mathbb{R}^k\hookrightarrow T_{\rm{GH}}\mbox{X}$ is the natural inclusion, for every $k\in\N$.

In other words, a subset $E$ of $T_{\rm{GH}}\mbox{X}$ belongs to $\mathcal{M}_{\rm{GH}}(\mbox{X})$ if and only if $E\cap(A_k\times\mathbb{R}^k)$ is a Borel subset of $A_k\times\mathbb{R}^k$ for every $k\in\N$.
\item[ii)] The projection $\pi:\,T_{\rm{GH}}\mbox{X}\to\mbox{X}$ of $T_{\rm{GH}}\mbox{X}$ is given by
\begin{equation}\label{f_projection_GH}
\pi(x,v):=x\quad\mbox{ for every }(x,v)\in T_{\rm{GH}}\mbox{X}\mbox{.}
\end{equation}
\item[iii)] The norm $\boldsymbol{\mathsf{n}}:\,T_{\rm{GH}}\mbox{X}\to[0,+\infty)$ on $T_{\rm{GH}}\mbox{X}$ is given by
\begin{equation}\label{f_norm_GH}
\boldsymbol{\mathsf{n}}(x,v):={|v|}_{\mathbb{R}^k}\quad\mbox{ for every }k\in\N \mbox{ and }(x,v)\in A_k\times\mathbb{R}^k\subseteq T_{\rm{GH}}\mbox{X}\mbox{.}
\end{equation}\vspace{-0.8cm}
\end{itemize}

\begin{definition}[Gromov-Hausdorff tangent bundle]\label{def_GH_tangent_bundle}
The \emph{Gromov-Hausdorff tangent bundle} of $(\emph{X},\mathsf{d},\mathfrak{m})$ is the  measurable Banach bundle
\begin{equation}\label{f_GH_tangent_bundle}
\big(T_{\rm{GH}}\rm{X},\mathcal{M}_{\rm{GH}}(\rm{X}),\pi,\boldsymbol{\mathsf{n}}\big).
\end{equation}
The space of the $L^2$-sections of such bundle is called \emph{Gromov-Hausdorff tangent module} and is denoted by $L^2(T_{\rm{GH}}\rm{X})$.
\end{definition}
The choice of this measurable structure on $T_{\rm{GH}}\mbox{X}$ could seem to be na\"{i}ve, but we now prove that it
is the only one coherent with some (thus any) atlas on $(\mbox{X},\mathsf{d},\mathfrak{m})$, in the sense which we now describe.

Let us fix an $\varepsilon$-atlas $\mathscr{A}=\big\{(U^k_i,\varphi^k_i)\big\}_{k,i}$ on $(\mbox{X},\mathsf{d},\mathfrak{m})$. For every $k,i\in\mathbb{N}$, choose a constant $C^k_i\geq 1$ such that
\begin{equation}
\label{eq:miscarte}
(C^k_i)^{-1}\,\mathcal{L}^k\restr{\varphi^k_i(U^k_i)}\leq(\varphi^k_i)_*(\mathfrak{m}\restr{U^k_i})
\leq C^k_i\,\mathcal{L}^k\restr{\varphi^k_i(U^k_i)}.
\end{equation}
Fix a sequence of radii $r_j\downarrow 0$ and define $\widehat{\varphi}^k_{ij}:\,U^k_i\times U^k_i\to A_k\times\mathbb{R}^k$ as
\begin{equation}\label{f_psi_hat}
\widehat{\varphi}^k_{ij}(\bar{x},x):=\left(\bar{x}\,,\,\frac{\varphi^k_i(x)-\varphi^k_i(\bar{x})}{r_j}\right)
\quad\mbox{ for every }(\bar{x},x)\in U^k_i\times U^k_i\mbox{.}
\end{equation}
For the sake of brevity, for $k,i,j\in\N$ let us call
\begin{equation}\label{f_W^k_ij}
\begin{split}
W^k_{ij}&:=\widehat{\varphi}^k_{ij}(U^k_i\times U^k_i),\\
W^k&:={\bigcup}_{i,j\in\mathbb{N}}\,W^k_{ij}
\end{split}
\end{equation}
and notice that simple computations yield
\begin{equation}\label{f_W^k_ij_biLip}
\begin{split}
&\widehat{\varphi}^k_{ij}:\,U^k_i\times U^k_i\to W^k_{ij}\quad\mbox{ is }\sqrt{1+(1+\varepsilon)^2/(r_j)^2}\,\mbox{-biLipschitz,}\\
&\frac{(r_j)^k}{C^k_i}\,(\mathfrak{m}\otimes\mathcal{L}^k)\restr{ W^k_{ij}}
\leq\big(\widehat{\varphi}^k_{ij}\big)_*\big((\mathfrak{m}\otimes\mathfrak{m})\restr{U^k_i\times U^k_i}\big)
\leq(r_j)^k C^k_i\,(\mathfrak{m}\otimes\mathcal{L}^k)\restr{ W^k_{ij}}\mbox{.}
\end{split}
\end{equation}
In particular, $W^k_{ij}\in\mathscr{B}(A_k\times\mathbb{R}^k)$ for every $k,i,j$, thus accordingly also $W^k\in\mathscr{B}(A_k\times\mathbb{R}^k)$. Put $N_k:=(A_k\times\mathbb{R}^k)\setminus W^k$.
\begin{lemma} With the notation just introduced, for every $k\in\N$ we have 
\[
(\mathfrak{m}\otimes\mathcal{L}^k)(N_k)=0.
\]
\end{lemma}
\begin{proof} For $k\in\N$ put 
\[
D_k:=\bigcup_{i\in\mathbb{N}}\big\{{x}\in U^k_i\,:\,\varphi^k_i({x})\text{ is a point of density 1 for }\varphi^k_i(U^k_i)\big\}.
\]
From \eqref{eq:miscarte} and \eqref{f_a.e._pt_1-density} we see that $\mathfrak{m}(A_k\setminus D_k)=0$, therefore for every $i,m,h\in\mathbb{N}$ and $\bar{x}\in D_k$, there is $j\in\mathbb{N}$ such that
$$1\geq\frac{\mathcal{L}^k\left(\frac{\varphi^k_i(U^k_i)-\varphi^k_i(\bar{x})}{r_j}\cap B_m(0)\right)}{\mathcal{L}^k\big(B_m(0)\big)}=\frac{\mathcal{L}^k\Big(\varphi^k_i(U^k_i)\cap B_{mr_j}\big(\varphi^k_i(\bar{x})\big)\Big)}{\mathcal{L}^k\Big(B_{mr_j}\big(\varphi^k_i(\bar{x})\big)\Big)}>1-\frac{1}{h}\mbox{,}$$
whence $\mathcal{L}^k\Big(B_m(0)\setminus\displaystyle{{\bigcup}_j}\big(\varphi^k_i(U^k_i)-\varphi^k_i(\bar{x})\big)/r_j\Big)=0$
for all $i,m\in\mathbb{N}$ and $\bar{x}\in D_k$. Therefore by Fubini's theorem we deduce
\[\begin{split}
(\mathfrak{m}\otimes\mathcal{L}^k)\Big(\big(A_k\times B_m(0)\big)\setminus W^k\Big)
=\,&\sum_{i\in\mathbb{N}}(\mathfrak{m}\otimes\mathcal{L}^k)\Big(\big(U^k_i\times B_m(0)\big)\setminus W^k\Big)\\
\leq\,&\sum_{i\in\mathbb{N}}\int_{D_k}\mathcal{L}^k\Big(B_m(0)\setminus{\bigcup}_j\big(\varphi^k_i(U^k_i)-\varphi^k_i(\bar{x})\big)/r_j\Big)\,\mbox{d}\mathfrak{m}(\bar{x})=0\mbox{,}
\end{split}\]
so that $(\mathfrak{m}\otimes\mathcal{L}^k)(N_k)=
\lim_m(\mathfrak{m}\otimes\mathcal{L}^k)\Big(\big(A_k\times B_m(0)\big)\setminus W^k\Big)=0$.
\end{proof}
We now endow $T_{\rm{GH}}\mbox{X}$ with a new $\sigma$-algebra $\mathcal{M}(\mathscr{A},(r_j))$, depending on the atlas $\mathscr{A}$ and the sequence $(r_j)$. Let  $\bar{\iota}_k:\,N^k\hookrightarrow T_{\rm{GH}}\mbox{X}$ be the inclusion maps, then define
\begin{equation}\label{f_M(A,r_j)}
\mathcal{M}\big(\mathscr{A},(r_j)\big):=\bigcap_{k\in\N}\bigg((\bar{\iota}_k)_*\mathscr{B}(N^k)\cap\bigcap_{i,j\in\mathbb{N}}(\iota_k\circ\widehat{\varphi}^k_{ij})_*\mathscr{B}(U^k_i\times U^k_i)\bigg)\mbox{.}
\end{equation}
Equivalently, a subset $E$ of $T_{\rm{GH}}\mbox{X}$ belongs to $\mathcal{M}(\mathscr{A},(r_j))$ if and only if $E\cap N^k\in\mathscr{B}(N^k)$ for every $k=1,\ldots,N$ and $(\widehat{\varphi}^k_{ij})^{-1}\big(E\cap(A_k\times\mathbb{R}^k)\big)\in\mathscr{B}(U^k_i\times U^k_i)$
for every $k,i,j$.

The fact that our choice of the  $\sigma$-algebra $\mathcal{M}_{\rm{GH}}(\rm{X})$ on $ T_{\rm{GH}}\rm{X}$ is canonical is encoded in the following proposition:
\begin{proposition}
Let $(\X,\sfd,\mm)$ be a strongly $\mm$-rectifiable metric measure space, $\mathscr A$ an $\eps$-atlas and $r_j\downarrow0$ a given sequence.

Then
\begin{equation}\label{f_M_GH_X=M(A,r_j)}
\mathcal{M}_{\rm{GH}}({\X})=\mathcal{M}\big(\mathscr{A},(r_j)\big)\mbox{.}
\end{equation}
\end{proposition}
\begin{proof}
If $E\in\mathcal{M}_{\rm{GH}}(\mbox{X})$ then $\iota^{-1}_k(E)\in\mathscr{B}(A_k\times\mathbb{R}^k)$ for every $k\in\N$, so accordingly $E\cap N^k$ belongs to $\mathscr{B}(N^k)$ and $(\widehat{\varphi}^k_{ij})^{-1}\big(\iota^{-1}_k(E)\big)$ belongs to $\mathscr{B}(U^k_i\times U^k_i)$ for every $k,i,j$, which proves that $E\in\mathcal{M}\big(\mathscr{A},(r_j)\big)$.

Conversely, let $E\in\mathcal{M}\big(\mathscr{A},(r_j)\big)$. Hence $E\cap N^k\in\mathscr{B}(N^k)\subseteq\mathscr{B}(A_k\times\mathbb{R}^k)$, while $F^k_{ij}:=(\widehat{\varphi}^k_{ij})^{-1}\big(\iota^{-1}_k(E)\big)\in\mathscr{B}(U^k_i\times U^k_i)$ implies
that $E\cap W^k_{ij}=\widehat{\varphi}^k_{ij}(F^k_{ij})\in\mathscr{B}(A_k\times\mathbb{R}^k)$.
Thus $\iota^{-1}_k(E)=(E\cap N^k)\cup\bigcup_{i,j}(E\cap W^k_{ij})\in\mathscr{B}(A_k\times\mathbb{R}^k)$ for every $k\in\N$, which is equivalent to saying that $E\in\mathcal{M}_{\rm{GH}}(\mbox{X})$. \end{proof}
\begin{remark}{\rm This last proposition does not use the strong $\mm$-rectifiability of the space but only the $\mm$-rectifiability, as seen by the fact that we didn't consider a sequence of $\eps_n$-atlases. We chose this presentation because the reason for the introduction of the Gromov-Hausdorff tangent module is in the statement contained in the next section, which grants that the space of its sections is isometric to the abstract tangent module $L^2(T\X)$, a result which we have only for strongly $\mm$-rectifiable spaces.
}\fr\end{remark}

\section{Equivalence of \texorpdfstring{$L^2(T\rm{X})$}{L^2(TX)} and \texorpdfstring{$L^2(T_{\rm{GH}}\rm{X})$}{L^2(T_GH_X)}}
\label{sect_equiv_L^2(TX)_and_L^2(T_GH_X)}
The main result of this article is the following: the two different notions of tangent modules described so far,
namely the ``analytic'' tangent module $L^2(T{\rm X})$ and the ``geometric'' Gromov-Hausdorff tangent module
$L^2(T_{\rm GH}{\rm X})$, can be actually identified. More precisely, given a strongly $\mathfrak{m}$-rectifiable
space that is also a PI space, there exists an isomorphism between $L^2(T{\rm X})$ and $L^2(T_{\rm GH}{\rm X})$ which preserves the pointwise norm and, as the construction, such isomorphism can be canonically chosen once an aligned sequence of atlases is given.

Notice that  Theorem \ref{thm_dim_decompos_of_L^2(TX)} (which is valid on more general $\mm$-rectifiable spaces) is equivalent to the fact that there exists a morphism of $L^2(T{\rm X})$ into $L^2(T_{\rm GH}{\rm X})$ with continuous inverse, thus in particular changing the pointwise norm of a bounded factor. Thus Theorem \ref{thm_L^2(TX)=L^2(T_GH_X)} below can be seen as the improvement of Theorem \ref{thm_dim_decompos_of_L^2(TX)} which shows that for strongly $\mm$-rectifiable spaces such factor can be taken to be 1.

\begin{theorem}[Equivalence of $L^2(T{\rm X})$ and $L^2(T_{\rm GH}{\rm X})$]\label{thm_L^2(TX)=L^2(T_GH_X)}
Let $({\rm X},\mathsf{d},\mathfrak{m})$ be a  strongly $\mathfrak{m}$-rectifiable and PI space. Then there exists an isometric isomorphism of modules
$\mathscr{I}:\,L^2(T{\rm X})\to L^2(T_{\rm GH}{\rm X})$, so that in particular it holds
\begin{equation}\label{f_L^2(TX)=L^2(T_GH_X)_1}
\big|\mathscr{I}(\mathsf{v})\big|=|\mathsf{v}|\;\;\mathfrak{m}\mbox{-a.e. in }\emph{X}\mbox{, }\quad\mbox{for every }\mathsf{v}\in L^2(T\emph{X})\mbox{.}
\end{equation}
\end{theorem}
\begin{proof}
Consider an aligned family $(\mathscr{A}_n)_n$ of atlases $\mathscr{A}_n={\big\{(U^{k,n}_i,\varphi^{k,n}_i)\big\}}_{k,i}$ on $(\mbox{X},\mathsf{d},\mathfrak{m})$, of parameters $\varepsilon_n:=1/2^n$ and $\delta_n:=1/2^{n}$, whose existence is guaranteed by Theorem \ref{thm_existence_aligned_atlases}. Now let $\mathsf{v}\in L^2(T{\rm X})$ and $n\in\mathbb{N}$ be fixed. For $k,i\in\N$ put $V^{k,n}_i:=\varphi^{k,n}_i\big(U^{k,n}_i\big)\in\mathscr{B}(\mathbb{R}^k)$ and recall  that $\varphi^{k,n}_i:\,U^{k,n}_i\to V^{k,n}_i$ and its inverse are maps of bounded deformation. Thus it makes sense to consider $\mbox{d}\varphi^{k,n}_i\big(\nchi_{U^{k,n}_i}\mathsf{v}\big)\in L^{2}\big(V^{k,n}_i,\mathbb{R}^k\big)\cong{L^2(T\mathbb{R}^k)}\restr{V^{k,n}_i}$ and we can  define
$$\mathsf{w}^{k,n}_i(x):=
\left\{\begin{array}{ll}
\big(\mbox{d}\varphi^{k,n}_i\big(\nchi_{U^{k,n}_i}\mathsf{v}\big)\big)\big(\varphi^{k,n}_i(x)\big)\\
0
\end{array}\quad\begin{array}{ll}
\mbox{ for }\mathfrak{m}\mbox{-a.e. }x\in U^{k,n}_i\mbox{,}\\
\mbox{ for }\mathfrak{m}\mbox{-a.e. }x\in\mbox{X}\setminus U^{k,n}_i\mbox{.}
\end{array}\right.$$
The bound \eqref{f_differential_bdd_deform_map_2} gives
\begin{equation}\label{f_L^2(TX)=L^2(T_GH_X)_2}
\big|\mathsf{w}^{k,n}_i\big|(x)\leq\mbox{Lip}(\varphi^{k,n}_i)\,|\mathsf{v}|(x)\quad\mbox{ for }\mathfrak{m}\mbox{-a.e. }x\in U^{k,n}_i\mbox{,}
\end{equation}
so that ${\big\|\mathsf{w}^{k,n}_i\big\|}_{L^2(T_{\rm GH}{\rm X})}\leq(1+2^{-n}){\big\||\mathsf{v}|\big\|}_{L^2(U^{k,n}_i)}$.
In particular, the series $\sum_{i,k}\mathsf{w}^{k,n}_i$ converges in $ L^{2}(T_{\rm GH}{\rm X})$ to some vector field $\mathscr{I}_n(\mathsf{v})$ whose norm is bounded by $\left(1+\frac{1}{2^n}\right){\big\||\mathsf{v}|\big\|}_{L^2({\rm X})}$ and which satisfies
\begin{equation}\label{f_L^2(TX)=L^2(T_GH_X)_3}
\nchi_{U^{k,n}_i}\mathscr{I}_n(\mathsf{v})=\mathsf{w}^{k,n}_i\quad\mbox{ for every }k,i\in\mathbb{N}\mbox{.}
\end{equation}
It is then clear that $\mathscr I_n:L^2(T\X)\to L^2(T_{\rm GH}{\rm X})$ is $L^\infty$-linear, continuous and satisfying  $\big|\mathscr{I}_n(\mathsf{v})\big|\leq\big(1+2^{-n}\big)\,|\mathsf{v}|$  $\mathfrak{m}$-a.e.\ for every ${\sf v}\in L^2(T\X)$.

We now claim that
\begin{equation}\label{f_L^2(TX)=L^2(T_GH_X)_4}
\mbox{the sequence }{(\mathscr{I}_n)}_n\mbox{ is Cauchy w.r.t.\ the operator norm.}
\end{equation}
To prove this, let ${\sf v}\in L^2(T\X)$,  $k,i,j\in\mathbb{N}$ with $U^{k,n+1}_i\subseteq U^{k,n}_j$. For $\mathfrak{m}$-a.e. point $x\in U^{k,n+1}_i$, putting for brevity $y:=\varphi^{k,n+1}_i(x)$, it holds that
\[\begin{split}
\big|\mathscr{I}_{n+1}(\mathsf{v})-\mathscr{I}_n(\mathsf{v})\big|(x)=\,&\bigg|\Big(\mbox{d}\varphi^{k,n+1}_i\big(\nchi_{U^{k,n+1}_i}\mathsf{v}\big)\Big)\big(\varphi^{k,n+1}_i(x)\big)-
\Big(\mbox{d}\varphi^{k,n}_j\big(\nchi_{U^{k,n+1}_i}\mathsf{v}\big)\Big)\big(\varphi^{k,n}_j(x)\big)\bigg|\\
(\eqref{eq:locdiff},\eqref{f_functoriality_differential},\eqref{f_differentials_in_R^k})\quad\leq\,&\bigg\|\,{\mbox{d}}\Big(\mbox{id}_{V^{k,n+1}_i}-\varphi^{k,n}_j\circ\big(\varphi^{k,n+1}_i\big)^{-1}\Big)(y)
\bigg\|\,\Big|\mbox{d}\varphi^{k,n+1}_i\big(\nchi_{U^{k,n+1}_i}\mathsf{v}\big)\Big|(y)\\
(\delta_{n+1}=2^{-n-1})\quad\leq\,&\frac{1}{2^{n+1}}\,\Big|\mbox{d}\varphi^{k,n+1}_i\big(\nchi_{U^{k,n+1}_i}\mathsf{v}\big)\Big|\big(\varphi^{k,n+1}_i(x)\big)\\
(\eps_{n+1}=2^{-n-1})\quad\leq\,&\frac{1}{2^{n+1}}\,\left(1+\frac{1}{2^{n+1}}\right)\,|\mathsf{v}|(x)\leq\frac{1}{2^n}\,|\mathsf{v}|(x)\mbox{.}
\end{split}\]
It follows that ${\big\|\mathscr{I}_{n+1}(\mathsf{v})-\mathscr{I}_n(\mathsf{v})\big\|}_{L^2(T_{\rm GH}{\rm X})}\leq\frac{1}{2^n}\,{\|\mathsf{v}\|}_{L^2(T{\rm X})}\mbox{}$ which by the arbitrariness of ${\sf v}$ means that
\[
\big\|\mathscr{I}_{n+1}-\mathscr{I}_n\big\|\leq\frac{1}{2^n},
\]
where the norm in the left hand side is the operator one. Hence $\sum_{n=0}^\infty\big\|\mathscr{I}_{n+1}-\mathscr{I}_n\big\|<+\infty$ and the   claim \eqref{f_L^2(TX)=L^2(T_GH_X)_4} is proved.

Let $\mathscr I:L^2(T\X)\to L^2(T_{\rm GH}\X)$ be the limit of  $(\mathscr I_n)$ and notice that being the limit of $L^\infty$-linear maps, it is also $L^\infty$-linear. Moreover, the fact that $\mathscr{I}_n(\mathsf{v})\to\mathscr{I}(\mathsf{v})$ in $L^2(T_{\rm GH}{\rm X})$ implies that $\big|\mathscr{I}_n(\mathsf{v})\big|\to\big|\mathscr{I}(\mathsf{v})\big|$ in $L^2({\rm X})$, hence - up to subsequences -  we have
\begin{equation}
\label{eq:normi}
\big|\mathscr{I}(\mathsf{v})\big|(x)=\lim_{n\to\infty}\big|\mathscr{I}_{n}(\mathsf{v})\big|(x)\leq\lim_{n\to\infty}\left(1+\frac{1}{2^{n}}\right)|\mathsf{v}|(x)=|\mathsf{v}|(x)\quad\mbox{ for }\mathfrak{m}\mbox{-a.e. }x\in\mbox{X}\mbox{.}
\end{equation}
In order to prove that $\mathscr{I}$ is actually an isometric isomorphism that preserves the pointwise norm, we explicitly exhibit its inverse functional $\mathscr{J}$. In analogy with the construction just done, for any $\mathsf{w}\in L^2(T_{\rm GH}{\rm X})$ and $n\in\mathbb{N}$ one can build a unique $\mathscr{J}_n(\mathsf{w})\in L^2(T{\rm X})$ such that
\begin{equation}\label{f_L^2(TX)=L^2(T_GH_X)_5}
\nchi_{U^{k,n}_i}\mathscr{J}_n(\mathsf{w})=\mbox{d}\big(\varphi^{k,n}_i\big)^{-1}\Big(\mathsf{w}\circ\big(\varphi^{k,n}_i\big)^{-1}\Big)\quad\mbox{ for every }k,i\in\mathbb{N}\mbox{.}
\end{equation}
By means of the same arguments used above, we can prove that $\mathscr J_n:L^2(T_{\rm GH}\X)\to L^2(T\X)$ is $L^\infty$-linear and continuous and, as $n\to\infty$, converges to a limit functional $\mathscr J:L^2(T_{\rm GH}\X)\to L^2(T\X)$ in the operator norm which is $L^\infty$-linear and satisfies
\begin{equation}
\label{eq:normj}
\big|\mathscr{J}(\mathsf{w})\big|\leq|\mathsf{w}|\quad \mathfrak{m}-a.e.\qquad\forall \mathsf{w}\in L^2(T_{\rm GH}{\rm X}).
\end{equation}
Our aim is now to prove that $\mathscr{I}$, $\mathscr{J}$ are one the inverse of the other. 

Let $\mathsf{v}\in L^2(T\mbox{X})$ and $n\in\mathbb{N}$ be fixed. For any $k,i\in\mathbb{N}$, we have that \eqref{f_L^2(TX)=L^2(T_GH_X)_3} and \eqref{f_L^2(TX)=L^2(T_GH_X)_5} give
\[\begin{split}
\nchi_{U^{k,n}_i}\mathscr{J}_n\big(\mathscr{I}_n(\mathsf{v})\big)&=
\nchi_{U^{k,n}_i}\mathscr{J}_n\big(\nchi_{U^{k,n}_i}\mathscr{I}_n(\mathsf{v})\big)=\nchi_{U^{k,n}_i}
\mathscr{J}_n\Big(\mbox{d}\varphi^{k,n}_i\big(\nchi_{U^{k,n}_i}\mathsf{v}\big)\circ\varphi^{k,n}_i\Big)\\
&=\mbox{d}\big(\varphi^{k,n}_i\big)^{-1}\Big(\mbox{d}\varphi^{k,n}_i\big(\nchi_{U^{k,n}_i}\mathsf{v}\big)\Big)=
\nchi_{U^{k,n}_i}\mathsf{v}\mbox{,}
\end{split}\]
therefore $\mathscr{J}_n\circ\mathscr{I}_n=\mbox{id}_{L^2(T{\rm X})}$. In an analogous way, also $\mathscr{I}_n\circ\mathscr{J}_n=\mbox{id}_{L^2(T_{\rm GH}{\rm X})}$. Thus for every  $n\in\mathbb{N}$ we have
\[\begin{split}
{\big\|\mathscr{J}\circ \mathscr{I}-{\rm id}_{L^2(T\X)}\big\|}=&{\big\|\mathscr{J}\circ \mathscr{I}-\mathscr{J}_n\circ \mathscr{I}_n\big\|}\\
\leq\,&{\big\|\mathscr{J}\circ\big(\mathscr{I}-\mathscr{I}_n\big)\big\|}+
{\big\|(\mathscr{J}-\mathscr{J}_n\big)\circ \mathscr{I}_n\big\|}\\
\leq\,&\|\mathscr{J}\|\, \big\|\mathscr{I}-\mathscr{I}_n\big\|+
{\big\|\mathscr{J}-\mathscr{J}_n\big\|}\,\sup_n \|\mathscr{I}_n\|,
\end{split}\]
so that letting $n\to\infty$ we see that $\mathscr{J}\circ \mathscr{I}={\rm id}_{L^2(T\X)}$. A symmetric argument yields $\mathscr{I}\circ\mathscr{J}=\mbox{id}_{L^2(T_{\rm GH}{\rm X})}$.

To conclude, notice that for every ${\sf v}\in L^2(T\X)$ we have
$$|\mathsf{v}|=\big|\mathscr{J}\big(\mathscr{I}(\mathsf{v})\big)\big|\stackrel{\eqref{eq:normj}}\leq\big|\mathscr{I}(\mathsf{v})\big|\stackrel{\eqref{eq:normi}}\leq|\mathsf{v}|\quad\mathfrak{m}\mbox{-a.e. in }\mbox{X}.$$
Hence the inequalities are equalities, yielding  \eqref{f_L^2(TX)=L^2(T_GH_X)_1} and the conclusion.
\end{proof}
\section{Geometric interpretation of \texorpdfstring{$T_{\rm{GH}}\rm{X}$}{T_GH_X}}
\label{sect_Geometric_Interpretation_of_T_GH_X}
The aim of this conclusive section is to discuss in which sense for strongly $\mm$-rectifiable spaces  the space $T_{\rm GH}\rm{X}$ can be obtained by looking at the pointed Gromov-Hausdorff limits of the rescalings of $\rm{X}$ around (almost) all of its points.

For comparison purposes, let us recall the definition of pointed  Gromov-Hausdorff convergence (see e.g.\ \cite{BBI01}):
\begin{definition}[Pointed Gromov-Hausdorff convergence]\label{def_pGH_convergence}
Let $({\rm X}_m,\mathsf{d}_m,\bar{x}_m)$, for any $m\in\mathbb{N}\cup\{\infty\}$, be a pointed metric space.
Then $(\emph{X}_m,\mathsf{d}_m,\bar{x}_m)$ is said to converge as $m\to\infty$ to $(\emph{X}_\infty,\mathsf{d}_\infty,\bar{x}_\infty)$ \emph{in the pointed Gromov-Hausdorff sense} provided the following property is satisfied: for every $0<\varepsilon<R$, there exist $\bar{m}\in\mathbb{N}$ and Borel maps $f_m:\,B_R(\bar{x}_m)\to\emph{X}_\infty$ for any $m\geq\bar{m}$ such that
\begin{itemize}
\item[\rm(i)] $f_m(\bar{x}_m)=\bar{x}_\infty$ for every $m\geq\bar{m}$,
\item[\rm(ii)] $f_m$ is an \emph{$\varepsilon$-quasi isometry} with its image for every $m\geq\bar{m}$, i.e.
\begin{equation}\label{f_quasi_isometry}
\begin{split}
&\Big|\mathsf{d}_{\infty}\big(f_m(x),f_m(x')\big)-\mathsf{d}_{m}(x,x')\Big|\leq\varepsilon\quad\mbox{ for every }x,x'\in B_R(\bar{x}_m)\mbox{,}\\
&B_{R-\varepsilon}(\bar{x}_\infty)\subseteq \eps-\text{neighbourhood of }{f_m\big(B_R(\bar{x}_m)\big)}\mbox{.}
\end{split}
\end{equation}
\end{itemize}
\end{definition}
Let now $(\X,\sfd,\mm)$ be a strongly $\mm$-rectifiable space and $\mathscr A_n=\{(U^{k,n}_i,\varphi^{k,n}_i)\}_{i,k}$ be a family of $\eps_n$-atlases,
with compact domains $U^{k,n}_i$. We can use the atlases to build Borel maps $\Psi_n:\X\times(\frac1{r_n}\X)\to T_{\rm GH}\X$ which are `bundle maps', i.e.\ which fix the first coordinate, and that are approximate isometries as maps on the second variable in the following way. We first recall that for any $U\subset \X$ closed there exists a Borel map $P_U:\X\to U$ such that 
\[
\sfd(x,P_U(x))\leq 2 \sfd(x,U)\qquad\forall x\in \X.
\]
This can be built by first considering a countable dense subset $\{x_n\}_n$ of $U$ and then
by declaring $P_U(x):=x$ for $x\in U$ and for $x\notin U$ defining
\[
P_U(x):=\{x_n\ :\ n\text{ is the least $m$ such that }\sfd(x,x_m)\leq 2\sfd(x,U)\}.
\]
Then given a sequence $r_n\downarrow 0$ we put
\begin{equation}
\label{eq:phin}
\Phi_n(x,y):=\frac{\varphi^{k,n}_i(P_{U^{k,n}_i}(y))-\varphi^{k,n}_i(x)}{r_n}\in\R^k\qquad\text{ for }x\in U^{k,n}_i,\ y\in \X
\end{equation}
and $\Phi_n(x,y):=0_{\R^k}$ if $x\in A_k\setminus\cup_{i}U^{k,n}_i$, where $(A_k)$ is the dimensional decomposition of $\X$. Finally we define
\[
\Psi_n(x,y):=\big(x,\Phi_n(x,y)\big)\qquad \forall x,y\in\X.
\]
Notice that $\Psi_n$ is Borel for every $n\in\N$. In the next theorem we show that for $\mm$-a.e.\ $x\in\X$ the maps $y\mapsto \Phi_n(x,y)$ provide approximate isometries of  $\X$  rescaled by a factor $\frac1{r_n}$ and $\R^k$, thus showing not only that the tangent space of $\X$ at $x$ is $\R^k$, but also that there is a `compatible' choice of approximate isometries making the resulting global maps, i.e.\  $\Psi_n$, Borel.
\begin{theorem}
Let $(\X,\sfd,\mm)$ be a strongly $\mm$-rectifiable space, $\eps_n\downarrow0$ and $\mathscr A_n=\{(U^{k,n}_i,\varphi^{k,n}_i)\}_{i,k}$ be a family of $\eps_n$-atlases
with compact domains $U^{k,n}_i$ and $(A_k)$ the dimensional decomposition of $\X$.

Then there exists a sequence $r_n\downarrow0$ such that, defining $\Phi_n$ as in \eqref{eq:phin}, for $\mm$-a.e.\ $x\in\X$ the following holds: for every $R>\eps>0$ there is $n_0\in\N$ so that for every $n\geq n_0$ we have
\begin{equation}
\label{eq:converg}
\begin{split}
\Big|\big|\Phi_n(x,y_0)-\Phi_n(x,y_1)|_{\R^k}-\frac{\sfd(y_0,y_1)}{r_n}\Big|\leq \eps\qquad \forall y_0,y_1\in B_{r_nR}(x),\\
B_{R-\eps}(0_{\R^k})\subset \eps\text{-neighbourhood of }\{\Phi_n(x,y)\ :\ y\in B_{r_nR}(x)\},
\end{split}
\end{equation}
where $k$ is such that $x\in A_k$.
\end{theorem}
\begin{proof} For every $k,i,n\in\N$  put $V^{k,n}_i:=\varphi^{k,n}_i(U^{k,n}_i)$ and notice that from \eqref{f_chart} we see that for $\mm$-a.e.\ $x\in U^{k,n}_i$ the point $\varphi^{k,n}_i(x)$ is of density 1 for $V^{k,n}_i$. Hence the set
\[
D:=\bigcap_{n\in\N}\bigcup_{i,k}\Big\{x\in U^{k,n}_i\ :\ x,\varphi^{k,n}_i(x)\text{ are points of density 1 for }U^{k,n}_i,\ V^{k,n}_i\text{ respectively}\Big\}
\]
is Borel and such that $\mm(\X\setminus D)=0$. Fix $\bar x\in D$ and $R>\eps>0$, let $k$ be such that $\bar x\in A_k$ and $i(n)$ such that $\bar x\in U^{k,n}_{i(n)}$.  

Fix $\bar\eps<\min\{\frac\eps{4R},\frac\eps{ R-\eps}\}$ positive  and repeatedly apply Lemma \ref{lemma_pts_1-density} to $\bar x,U^{k,n}_{i(n)}$ and to $\varphi^{k,n}_{i(n)}(\bar x), V^{k,n}_{i(n)}$ with $\bar\eps$ in place of $\eps$ to find a sequence $r_n\downarrow0$ such that for every $n\in\N$ it holds
\begin{equation}
\label{eq:dadens1}
\begin{split}
\sfd\big(y,P_{U^{k,n}_{i(n)}}(y)\big)&\leq2 \bar\eps r_nR\qquad\qquad\qquad\forall y\in B_{r_nR}(\bar x),\\
\sfd_{\R^k}\big(z,V^{k,n}_{i(n)}\big)&\leq\bar\eps |z-\varphi^{k,n}_{i(n)}(\bar x)|\qquad\ \forall z\in B_{r_nR}\big(\varphi^{k,n}_{i(n)}(\bar x)\big).
\end{split}
\end{equation}
From the fact that $\varphi^{k,n}_{i(n)}$ is $(1+\eps_n)$-biLipschitz we see that for any $ y_0,y_1\in B_{r_nR}(\bar x)$ it holds
\[
\begin{split}
\big|\Phi_n(\bar x,y_0)-\Phi_n(\bar x,y_1)\big|_{\R^k}&\leq \frac{1+\eps_n}{r_n}\sfd\big(P_{U^{k,n}_{i(n)}}(y_0),P_{U^{k,n}_{i(n)}}(y_1)\big)\\
(\text{by } \eqref{eq:dadens1})\qquad&\leq\frac{1+\eps_n}{r_n}\big(\sfd(y_0,y_1)+4\bar\eps r_nR\big).
\end{split}
\]
Similarly we get $\big|\Phi_n(\bar x,y_0)-\Phi_n(\bar x,y_1)\big|_{\R^k}\geq \frac{1}{(1+\eps_n)r_n}\big(\sfd(y_0,y_1)-4\bar\eps r_nR\big)$, thus
\[
\Big|\big|\Phi_n(\bar x,y_0)-\Phi_n(\bar x,y_1)\big|_{\R^k}-\frac{\sfd(y_0,y_1)}{r_n}\Big|\leq 2R \max\Big\{2(1+\eps_n)\bar\eps +\eps_n,\frac{2\bar\eps+\eps_n}{1+\eps_n}\Big\}\quad\forall  y_0,y_1\in B_{r_nR}(\bar x).
\]
Since $\bar\eps<\frac\eps{4R}$, this is sufficient to show that the first in \eqref{eq:converg} is fulfilled for $n$ large enough.

For the second, let $w\in\R^k$ be with $|w|<R-\eps$ and put $z_n:=\varphi^{k,n}_{i(n)}(\bar x)+r_nw$. Thus $z_n\in B_{r_nR}(\varphi^{k,n}_{i(n)}(\bar x))$ and from the second in \eqref{eq:dadens1} and the compactness of $U^{k,n}_{i(n)}$ we deduce that there exists $y_n\in U^{k,n}_{i(n)}$ such that
\begin{equation}
\label{eq:dz}
|z_n-\varphi^{k,n}_{i(n)}(y_n)|\leq\bar\eps  r_n|w|.
\end{equation}
Since the right hand side is bounded from above by $\bar\eps r_nR$, for $n$ sufficiently large it is bounded above by $\eps$, so that to conclude it is sufficient to show that, independently on the choice of $w$, for $n$ sufficiently large it holds $y_n\in B_{r_nR}(\bar x)$. To see this, recall  that the inverse of $\varphi^{k,n}_{i(n)}$ is $(1+\eps_n)$-Lipschitz to get
\[
\begin{split}
\sfd(\bar x,y_n)&\leq(1+\eps_n)|\varphi^{k,n}_{i(n)}(\bar x)-\varphi^{k,n}_{i(n)}(y_n)|\leq (1+\eps_n)\big(|\varphi^{k,n}_{i(n)}(\bar x)-z_n|+|z_n-\varphi^{k,n}_{i(n)}(y_n)|\big)\\
\text{by }\eqref{eq:dz}\qquad&\leq r_n(1+\eps_n)(1+\bar \eps)|w|\leq r_n(1+\eps_n)(1+\bar \eps)(R-\eps).
\end{split}
\]
Since $\bar\eps<\frac{\eps}{R-\eps}$ we have $(1+\bar \eps)(R-\eps)<R$ and therefore for $n$ sufficiently large we have $ r_n(1+\eps_n)(1+\bar \eps)(R-\eps)< r_nR$, which concludes the proof.
\end{proof}
\def\cprime{$'$} \def\cprime{$'$}

\end{document}